\title{On rings of differential operators derived from  automorphic forms}
\author{Atsuhira Nagano}
\documentclass[10pt]{article}
\usepackage{mathrsfs,bm,graphics,amsfonts,amsthm,amssymb,amsmath,amscd,booktabs}
\usepackage[dvips]{graphicx,color}
\usepackage{mathrsfs,bm,graphics}
\def\bigzerou{\smash{\lower1.7ex\hbox{\b 0}}}

\newtheorem{thm}{Theorem}[section]
\newtheorem{df}{Definition}[section]
\newtheorem{lem}{Lemma}[section]
\newtheorem{prop}{Proposition}[section]
\newtheorem{rem}{Remark}[section]
\newtheorem{exap}{Example}[section]

\newtheorem{cor}{Corollary}[section]

\def\comment#1{{ }}
\makeatletter
  
      \@addtoreset{equation}{section}
\begin{document}
\maketitle
\setlength{\baselineskip}{10 pt}

\begin{abstract}
We study linear ordinary differential equations
 which are
 analytically parametrized on Hermitian symmetric spaces
and  invariant under the action of  symplectic groups.
They are generalizations of the classical Lam\'e equation.
Our main result gives
a closed relation between such differential equations and automorphic forms for symplectic groups.
Our study is based on techniques concerning  with
 the monodromy of complex differential equations,
the Baker-Akhiezer functions
and algebraic curves attached to rings of differential operators.
\end{abstract}

\renewcommand{\thefootnote}{\fnsymbol{footnote}}
\footnote[0]{Running head: Differential operators derived from automorphic forms}
\vspace{2mm}

\footnote[0]{
Note:
This is the corrected version of the published article DOI: 10.1007/s11785-017-0663-7. 
Typos and misleading phrases are corrected here.
Especially, for simplicity, criteria in Section 1.5 and 2.7 are corrected using arithmetic genera of algebraic curves.
}
\setlength{\baselineskip}{14 pt}
\renewcommand{\thefootnote}{\arabic{footnote}}

\section*{Introduction}

The main purpose of this paper is  to study linear ordinary differential operators of a complex independent variable
which are analytically  parametrized on  Hermitian symmetric domains
and invariant under the action of symplectic groups.
Our main result gives a closed relation
between commutative rings of such differential operators
and automorphic forms for symplectic groups.

Let us start the introduction with a typical example of the differential equations we study: the Lam\'e differential equation
\begin{align}\label{LameDiff}
\Big( -\frac{\partial^2}{\partial z^2} + B \wp(\Omega,z) \Big) u=X u,
\end{align}
where $B,X \in \mathbb{C}$ and $\wp(\Omega,z)$ is the Weierstrass $\wp$-function with the double periods $1$ and $\Omega \in \mathbb{H} = \{z\in \mathbb{C} | {\rm Im} (z) > 0\}$.
The Lam\'e differential equation has the regular singular points at every $z_0\in \mathbb{Z} + \mathbb{Z} \Omega$.
If $B=\rho (\rho +1),$
the characteristic exponents  at every singular point are $\rho+1$ and $-\rho.$
When $\rho\in \mathbb{Z}_{> 0}$,
a system of basis of the space  of solutions of (\ref{LameDiff}) is generated by $\Lambda(z)$ and $\Lambda (-z)$.
Here,
$
\displaystyle
\Lambda (z) = \prod_{j=1}^\rho \frac{\sigma (\Omega,z+ \kappa_j)}{\sigma (\Omega,z) } e^{-z \zeta (\Omega,\kappa_j)},
$
where 
$\sigma(\Omega,z)$ and $\zeta(\Omega,z)$ are 
 the classical Weierstrass functions
 and $\kappa_j$ $(j=1,\cdots,\rho)$ can be calculated  by $X$ (for detail, see \cite{WW}).
 We remark that
 $\Lambda$ is  a single-valued function of $z$.
 However,
 for generic $\rho \in \mathbb{C}$,
 the solutions of (\ref{LameDiff}) are  multivalued on $\mathbb{C}-(\mathbb{Z} + \mathbb{Z} \Omega).$
 The Lam\'e equation is an important topic in mathematics.
 For example, 
 the periodic solutions of (\ref{LameDiff}) is studied in many body theoretical physics.
Also,
via the double covering  $E\rightarrow \mathbb{P}^1(\mathbb{C})$,
where $E $ is an elliptic curve with the double periods $1$ and $\Omega$,
the equation (\ref{LameDiff}) gives a Fuchsian differential equation with an accessary parameter.
Moreover,
special types of (\ref{LameDiff})  promoted a development of the theory of integrable systems and finite zone problems.
For example, see \cite{WW}, \cite{DMN}, \cite{MM} and \cite{Takemura}.
In these researches,
to the best of the author's knowledge,
the double periodicity of the coefficient $\wp(\Omega,z)$ of (\ref{LameDiff}) played an essential role.

For our purpose,
we focus on another important property of the Lam\'e equation:
the coefficient $\wp(\Omega,z)$ satisfies the transformation law
\begin{align}\label{LameJacobi}
\wp\Big(\frac{a\Omega +b}{c\Omega +d},\frac{z}{c\Omega+d}\Big) = (c\Omega +d)^2 \wp(\Omega,z)
\end{align}
for any $\begin{pmatrix} a & b \\ c & d \end{pmatrix} \in SL(2,\mathbb{Z})$.
Due to (\ref{LameJacobi}),
the Lam\'e equation becomes to be invariant under the action of the elliptic modular group $SL(2,\mathbb{Z})$.
Namely,
via the transformation $\displaystyle (\Omega,z,X) \mapsto (\Omega_1,z_1,X_1)=\Big(\frac{a\Omega +b}{c\Omega +d},\frac{z}{c\Omega+d},(c\Omega +d)^2 X\Big),$
the differential equation (\ref{LameDiff}) can be identified with 
$
\displaystyle \Big( -\frac{\partial ^2}{\partial z_1^2} + B \wp(\Omega_1,z_1) \Big) u=X_1 u.
$
By the way,
holomorphic functions on $\mathbb{H}$ which are invariant under the action of $SL(2,\mathbb{Z})$ are called elliptic modular forms.
The invariance between two differential equations suggests a strong and non-trivial relation
between the Lam\'e equation and elliptic modular forms.
Furthermore,  elliptic modular forms are quite important in number theory (see \cite{S71}).
The author expects that the Lam\'e equation may have some effective applications in number theory.

Based on the above observation and expectation,
we study a class of ordinary differential equations $P u =Xu$ 
of a complex variable $z$ (for detail, see Definition \ref{DfDiff}).
Here,
the differential operator 
$$
P=
\frac{\partial^{N}}{\partial z^{N}} +a_2(\Omega,z) \frac{\partial^{N-2}}{\partial z^{N-2}} +a_3(\Omega,z)\frac{\partial ^{N-3}}{\partial z^{N-3}} +\cdots + a_N(\Omega,z)
$$
is parametrized by $\Omega$ of a product $\mathbb{H}_n^g$ of the Siegel upper half planes 
and invariant under the action of a congruence subgroup $\Gamma$ of the symplectic group.
Such a class contains the Lam\'e equation
because the action of the  group $\Gamma$
on  $\mathbb{H}_n^g$ is a natural extension of the action of $SL(2,\mathbb{Z})$ on $\mathbb{H}.$
In this paper,
we study  commutative rings of differential operators which commute with $P.$

Here, we recall the importance of  commutative rings of differential operators.
Commutative rings of differential operators  were firstly studied by
Burchnall and  Chaundy \cite{BC}.
In the later half of the 20th century,
the relation between
commutative rings of differential operators and algebraic curves 
was studied in the celebrated works of Krichever \cite{Kr} and Mumford \cite{Mum}.
Their results are very important
  in the theory of integrable systems.  
Also, they yielded a substantial progress
 of the geometry of Riemann surfaces and abelian varieties.
In fact,
they were used to resolve the classical Riemann-Schottky problem for Riemann surfaces (\cite{Shiota}, \cite{KS}).

In this paper,
we will give 
a relation between  commutative rings of differential operators and automorphic forms.
We study the structures of   rings  of differential operators 
which are invariant under the action of $\Gamma$
 and   commute with the fixed differential operator $P$.
Such a ring will be denoted by  $\mathcal{D}^P$ in Section 2.
Our main result gives an isomorphism $\chi:\mathcal{D}^P \simeq S^P$ of rings,
where $S^P$ is a ring of generating functions
for sequences of automorphic forms for $\Gamma$
(for the definition, see Definition \ref{dfDP} and \ref{dfSP}).
Here, we note that
automorphic forms are natural extension 
of elliptic modular forms (see Definition \ref{dfAutomorph}).
These rings are graded by the weight $K$ induced from the action of $\Gamma$:
$\displaystyle \mathcal{D}^P = \bigoplus_{K=0}^\infty \mathcal{D}^P_K , S^P = \bigoplus_{K=0}^\infty S^P_K .$
The isomorphism $\chi$ induces an isomorphisms among three vector spaces:
$$
\mathcal{D}_K^P \xrightarrow{\chi} S_K^P \xrightarrow{} W_K
$$
(see Theorem  \ref{ThmRef}).
Here, $W_K$ is a vector space 
explicitly parametrized by automorphic forms for $\Gamma$.
Therefore, 
the structure of the ring $\mathcal{D}^P$
is closely related to the structure  of  the rings of automorphic forms.

For our study,
we will use the Baker-Akhiezer functions.
In  \cite{Kr},
the 
 Baker-Akhiezer functions give  solutions of  differential equations whose coefficients are smooth functions.
However, for our purpose,
it is natural to study differential equations whose coefficients have poles (precisely, see Remark \ref{RemOur}).
So, we need to modify the techniques of the Baker-Akhiezer functions for  differential equations with some singularities.
Section 1 will be devoted to such  techniques.
Namely, we will study the multivalued Baker-Akhiezer functions and its monodromy around singular points of $P$.

In Section 2,
we prove our main result.
This is based on an invariance  of the multivalued Baker-Akhiezer functions under the action of $\Gamma$,
 which is proved in Theorem \ref{ThmPsi}.
Moreover, 
we will see the following results:
 \begin{itemize}

 \item 
 For  fixed $P$ and an operator $Q\in \mathcal{D}^P$,
 there exists an algebraic curve
$$
\displaystyle \mathcal{R}_\Omega: \displaystyle \sum_{j,k} f_{j,k} (\Omega) X^j Y^k =0
$$
such that $(X,Y)=(P,Q)$ gives a point of $\mathcal{R}_\Omega.$
Here,
the coefficients  $f_{j,k} (\Omega)$ are automorphic forms for $\Gamma$ (see Theorem \ref{ThmRiemannSurface}).
Namely,
from the differential operators $P$ and $Q$,
we  obtain a family of algebraic curves 
$\{\mathcal{R}_\Omega| \Omega\in\mathbb{H}_n^g\}$
 parametrized on $\mathbb{H}_n^g$ via automorphic forms.

\item
If the coefficients of the fixed operator $P$ have poles in $z$-plane,
the coefficients of $Q\in\mathcal{D}^P$ can be multivalued functions of $z$
(for detail, see Proposition \ref{PAsz} and Theorem \ref{PAs}).
However, 
if the genus of the algebraic curve $\mathcal{R}_\Omega$ is small enough,
every coefficients of $Q\in \mathcal{D}^P$ must be single-valued.
We will have a sufficient criterion
for $Q$
to be single-valued  (see Theorem \ref{ThmCriteriaAuto}).

\end{itemize}

Throughout the paper,
the Lam\'e differential equation is a  prototype of our story.
Via our new results between differential operators and automorphic forms,
we have a simple interpretation of classical results
of the Lam\'e equation via elliptic modular forms
(Example \ref{ExapLame}, \ref{ExapLameSpecial} and \ref{ExapLameRiemann}).
This is an important example of our story.

Our results enable us to study differential equations
based on the structures of rings of automorphic forms.
In number theory,
there are many 
famous generalizations of elliptic modular forms
(for example,
Siegel modular forms, Hilbert modular forms, etc.).
Our results can be applied to such generalized forms also.
The author expects that
this paper may give a first step
of the study
of differential equations from the viewpoint of automorphic forms.

\section{Commutative rings of differential operators with singularities and multivalued Baker-Akhiezer functions}

\subsection{Multivalued Baker-Akhiezer functions}
In this subsection,
we obtain the multivalued Baker-Akhiezer functions
for the ordinary differential operator  
\begin{align}\label{Pz}
P_{z}=
\frac{d^{N}}{d z^{N}} +a_2(z) \frac{d^{N-2}}{d z^{N-2}} +a_3(z)\frac{d ^{N-3}}{d z^{N-3}} +\cdots + a_N(z)
\end{align}
of the complex variable $z$.
Here, we assume the coefficients 
 $a_2(z), \cdots, a_N(z)$ are meromorphic functions of $z$.
 More precisely,  we assume that $a_2(z), \cdots, a_N(z)$ are holomorphic on  $\mathbb{C}-\mathcal{N}$,
 where  $\mathcal{N}$  is the union of the sets of the poles of $a_j(z)$ $(j=2,\cdots,N)$.

\begin{rem}
If a differential operator
$
P_z^0=\frac{d^{N}}{d z^{N}} + a_1^0 (z) \frac{d^{N-1}}{d z^{N-1}}  +a_2^0(z) \frac{d^{N-2}}{d z^{N-2}} +\cdots + a_N^0 (z) 
$
is given,
by a gauge transformation $v P_z^0 v^{-1}$ for some unit function $v=v(z)$,
$P_z^0$ is transformed to $P_z .$
So, in our study,
we only consider the differential operator in the form (\ref{Pz}) without loss of generality.
\end{rem}

Let $\mathfrak{X}$ be the universal covering of $\mathbb{C}-\mathcal{N}.$
By taking a fixed point $w\in \mathbb{C}-\mathcal{N}$,
any $s\in \mathfrak{X}$ is represented by
$
s=(z,[\gamma]),
$
where $z\in \mathbb{C}-\mathcal{N}$, $\gamma$ is an arc in $\mathbb{C}-N$ from $w$ to $z$ and
 $[\gamma]$ is the homotopy class of $\gamma$.
We note that $z $ gives a local coordinate of $\mathfrak{X}$.

\begin{prop}\label{PropPsiz}
There exists the unique formal solution 
$\Psi((z,[\gamma]),w,\lambda)$
of the differential equation
\begin{align}\label{sp-diffz}
P_{z} u = \lambda^N u
\end{align}
in the form
\begin{align}\label{Psiz}
\Psi((z,[\gamma]),w,\lambda )=\Big(\displaystyle \sum_{s=0}^\infty  \xi_s((z,[\gamma]),w) \lambda^{-s}\Big) e^{\lambda(z-w)}
\end{align}
such that
\begin{align}\label{ConditionPsiz}
\begin{cases}
& \xi_0((z,[\gamma]),w)\equiv 1,\\
& \xi_s((w,[id]),w) = 0 \quad (s\geq 1).
\end{cases}
\end{align}
Here, $\xi_s$ are locally holomorphic functions  of $(z,w)$. 
 \end{prop}

\begin{proof}
In this proof, set $a_0(z)\equiv 1,a_1(z)\equiv 0$.
Putting $u=\Big(\displaystyle \sum_{s=0}^\infty \eta_s(z) \lambda^{-s} \Big) e^{\lambda (z-w)}$ to (\ref{sp-diffz}), 
we have
\begin{align*}
\displaystyle \sum_{m=0}^N a_{N-m} (z) \displaystyle \sum_{l=0}^{m} \begin{pmatrix} m \\ l \end{pmatrix} \displaystyle \sum_{s=0}^\infty \Big( \frac{\partial^{m-l}}{\partial z^{m-l}} \eta_s(z) \lambda^{l-s} \Big)e^{\lambda(z-w)}
= \Big( \displaystyle \sum_{s=0}^\infty \eta_s(z)\lambda^{N-s}\Big)e^{\lambda(z-w)}.
\end{align*}
Comparing the coefficients of $\lambda^{-s_0}$, we have
\begin{align}\label{CoeffXiz}
\displaystyle \sum_{m=0}^N a_{N-m} (z) \displaystyle \sum_{l=0}^m\begin{pmatrix} m\\ l  \end{pmatrix} \frac{\partial ^{m-l}}{\partial z^{m-l}} \eta_{l+s_0} (z) =\eta_{N+s_0 } (z).
\end{align}
Since $\eta_{N+s_0}(z)$ appears in the left hand side only when $m=l=N$,
the terms of $\eta_{N+s_0} (z)$ is cancelled from the relation (\ref{CoeffXiz}).
The function $\eta_{N+s_0-1}(z) $ and its derivation appears in (\ref{CoeffXiz}) only when 
$m=N$ and $l=N-1$.
Here, we used $a_{N-1}(z)\equiv 0$.
Then, the equation (\ref{CoeffXiz}) becomes to be
\begin{align}\label{Inductionz}
N\frac{\partial }{\partial z} \eta_{N+s_0-1} (z)= \left(\text{a polynomial in } \frac{\partial^\nu }{\partial z^\nu}\eta_l (z) \hspace{1mm} (l<N+s_0-1, \nu \in \mathbb{Z}_{\geq 0})\text{ and } a_j (z) \text{ defined over } \mathbb{Z}\right).
\end{align}
By the integration of the relation (\ref{Inductionz}) on the arc $\gamma \in \mathbb{C}-\mathcal{N}$ whose start point is $w$,
we can  obtain the expression of $\eta_\mu (z)$ in terms of $\eta_\nu (z)$ $(\nu < \mu)$ and $a_l(z)$.  
Especially,
the condition that $\eta_0(z,[\gamma])\equiv 1$
and
$\eta_s(w,[id])=0$ $(s\geq1)$
uniquely determines the sequence $\{\eta_s(z)\}_s$.
Such functions $\eta_s (z)$ give the required functions $\xi_s((z,[\gamma]),w)$ $(s\geq 0)$.

From our construction given by the integration of the relation (\ref{Inductionz}), 
we can see that  $\xi_s$ are locally holomorphic functions of $(z,w)$.
\end{proof}

We call $\Psi((z,[\gamma]),w,\lambda)$ of (\ref{Psiz}) the multivalued Baker-Akhiezer function for the equation (\ref{sp-diffz}). 

\begin{rem}\label{RemOur}
 Krichever \cite{Kr}
 studied 
ordinary differential equations 
whose coefficients are smooth  functions of a real variable.
Also,
Mumford \cite{Mum}
studied differential equations whose coefficients are formal power series.
For the purposes of their research,
it is sufficient to study
 single-valued solutions of differential equations.
However,
for our main purpose of this paper,
it is natural to study 
differential equations of a complex independent variables
whose coefficients allow some singularities.
In Section 2,
we will consider the transformation
$z\mapsto z_1=\frac{z}{j_\alpha (\Omega)}$,
where $ j_\alpha (\Omega)$ is complex valued.
In such cases,
even if $z$ is a real variable,
$z_1$ is not always a real variable.
Moreover,
we will give results for a class of differential equations containing the Lam\'e equation.
Since the Lam\'e equation has singularities,
it is natural to study differential equations which admit singularities.
 This is the reason why we need the multivalued solution $z\mapsto \Psi ((z,[\gamma]),w,\lambda)$ of (\ref{Psiz}).
  \end{rem}

\begin{rem}
If $a_2(z),\cdots,a_N(z)$  are holomorphic on the whole $z $-plane,
we do not need to consider the universal covering $\mathfrak{X}$ of $\mathbb{C}-\mathcal{N}$.
In this case,
the function $\Psi$ in the above theorem is given in the form
$$
\Psi(z,w,\lambda )=\Big(\displaystyle \sum_{s=0}^\infty  \xi_s(z,w) \lambda^{-s}\Big) e^{\lambda(z-w)}.
$$
Here, $\Psi$ and $\xi_s$ $(s\geq 0)$ are single-valued functions of $z\in \mathbb{C}$.
\end{rem}

For a differential operator (\ref{Pz}), consider the differential equation 
\begin{align}\label{DiffXz}
P_{z} u = X u,
\end{align}
where $X\in \mathbb{P}^1(\mathbb{C})-\{\infty\}.$
Let $\lambda_1,\cdots,\lambda_N$ be the solutions of the equation  $\lambda^N=X$.

\begin{lem}\label{LemIndp}
For fixed $w\in \mathbb{C}-\mathcal{N}$, the solutions $\Psi((z,[\gamma]),w,\lambda_j)$ ($j=1,\cdots,N$)   of (\ref{Psiz})
are linear independent for generic $X$.
\end{lem}

\begin{proof}
For $\mu_1,\cdots,\mu_N \in \mathbb{C},$ suppose 
\begin{align}\label{linearPsiz}
\displaystyle \sum_{j=1}^N \mu_j \Psi ((z,[\gamma]),w,\lambda_j)=0
\end{align}
 holds for generic $X$.
 Since the right hand side of the relation (\ref{linearPsiz}) is invariant under the permutation of $\lambda_1,\cdots,\lambda_N$,
 together with the definition of $\Psi$ of (\ref{Psiz}),
 we can assume that $\mu_1=\cdots=\mu_N=\mu$.
 Set $E_s(z,w,X)=\displaystyle \sum_{j=1}^N \lambda_j^{-s} e^{\lambda_j (z-w)}$.
 For fixed $z$ and $w$,
 $X\mapsto E_s(z,w,X)$ is a formal power series in $X^{-1}$ and $E_s(z,w,X)$ $(s=0,1,\cdots)$ are linearly independent.
   The relation (\ref{linearPsiz}) becomes
   $\displaystyle \sum_{s=0}^\infty \mu \xi_s((z,[\gamma]),w) E_s(z,w,X)=0$ for generic $X$. 
   Therefore, $\mu=0$ follows. 
\end{proof}

\begin{prop}\label{PropUPAz}
Let $u=u((z,[\gamma]),w)$ be a series given by the form
\begin{align}\label{uz}
u((z,[\gamma]),w)= \left( \displaystyle \sum_{s=0}^\infty \eta_s(z,[\gamma]) \lambda^{-s}\right) e^{\lambda (z-w)},
\end{align}
where $\eta_s(z,[\gamma])$ are analytic on $\mathfrak{X}$ and  $\lambda$ satisfies $\lambda^N =X$.
Then, $u$ is a formal solution of the differential equation (\ref{DiffXz}) 
if and only if
$u$ is given by
\begin{align}\label{uAz}
u((z,[\gamma]),w)=A(w,\lambda) \Psi((z,[\gamma]),w,\lambda)
\end{align}
for generic $\lambda$, where $A(w,\lambda)$ does not depend on $(z,[\gamma])$.
\end{prop}

\begin{proof}
It is clear that $u((z,[\gamma]),w)$ of
(\ref{uAz}) is a solution of the differential equation (\ref{DiffXz}).

Conversely,
we assume that $u((z,[\gamma]),w)$ in the form (\ref{uz}) is a solution of the differential equation (\ref{DiffXz}),
where $X=\lambda^N$.
From Lemma \ref{LemIndp},
the space of solutions of (\ref{DiffXz}) is generated by $\Psi((z,[\gamma]),w,\lambda_j)$ $(j=1,\cdots,N)$ for generic $\lambda$.
We can assume $\lambda$ of (\ref{uz}) coincides with $\lambda_j$ for some $j\in \{1,\cdots, N\}$.
Since the space $\mathfrak{X}$ is simply connected,
$(z,[\gamma])\mapsto \Psi((z,[\gamma]),w,\lambda_j)$ $(j=1,\cdots,N)$ are single-valued on $\mathfrak{X}$.
So,
the solution in the form (\ref{uz}) must be an element of the $1$-dimensional vector space generated by $\langle \Psi ((z,[\gamma]),w,\lambda_j) \rangle.$
Hence, $u$ is given by the form (\ref{uAz}).
\end{proof}

We will consider a differential operator
\begin{align}\label{Qz}
Q_{(z,[\gamma])}=b_0(z,[\gamma]) \frac{d^M}{d z^M} +b_{1}(z,[\gamma]) \frac{d^{M-1}}{d z^{M-1}} + \cdots + b_M(z,[\gamma]).
\end{align}
Here,
we assume that the coefficients  $b_{k} (z,[\gamma])$ $(k=0,\cdots,M)$ are multivalued analytic functions on $\mathbb{C}-\mathcal{N}$.
The operator $Q_{(z,[\gamma])}$ is   defined on $\mathfrak{X}$.

From now on,
we consider the action of the operator $Q_{(z,[\gamma])}$
on the function $\Psi((z,[\gamma]),w,\lambda)$.
If $P_z$ and $Q_{(z,[\gamma])}$ are commutative,
we can apply Proposition \ref{PropUPAz} to $Q _{(z,[\gamma])}\Psi((z,[\gamma]),w,\lambda)$.
Therefore,
it is natural to consider differential operator (\ref{Qz})
whose coefficients are multivalued functions  of $z$
(for detail,
see the proof of the next proposition).

\begin{prop}\label{PAsz}
Let $P_{z}$ ($Q_{(z,[\gamma])}$, resp.) be the differential operator of (\ref{Pz}) ((\ref{Qz}), resp.).
Then, $P_{z}$ and $Q_{(z,[\gamma])}$ are commutative if and only if 
the quotient
$\displaystyle
\frac{Q _{(z,[\gamma])}\Psi((z,[\gamma]),w,\lambda)}{\Psi((z,[\gamma]),w,\lambda)} $
coincides with 
\begin{align}\label{star}
A(\lambda)=\displaystyle \sum_{s=-M}^\infty A_s \lambda^{-s}
\end{align}
for generic $\lambda$,
where  $\Psi$ is given in (\ref{Psiz})
and $A(\lambda)$ does not depend on $(z,[\gamma])$ and $w$.
\end{prop}

\begin{proof}
Suppose that $P_{z}$ and $Q_{(z,[\gamma])}$ are commutative.
Then, $Q_{(z,[\gamma])} \Psi((z,[\gamma]),w,\lambda)$ gives a solution of the differential equation (\ref{DiffXz}).
Remark that
$Q_{(z, [ \gamma ])} (\Psi((z,[\gamma]),w,\lambda))$
is in the form (\ref{uz}) for some $\{\eta_s\}_s$.
So, due to Proposition   \ref{PropUPAz}, 
we have
\begin{align}\label{QPAPz}
\displaystyle
Q _{(z,[\gamma])}\Psi ((z,[\gamma]),w,\lambda) = A(w,\lambda) \Psi ((z,[\gamma]),w,\lambda),
\end{align}
for some $A(w,\lambda)$.
Take any $w' \in \mathbb{C}-\mathcal{N}$.
Then,
$\Psi ((z,[\gamma]),w',\lambda ) e^{\lambda(w'-w)}$ 
has the form (\ref{uz})
and is a solution of (\ref{DiffXz}).
So, according to Proposition $\ref{PropUPAz}$ again,
there exist $B(w,\lambda)$ such that
\begin{align}\label{PePBz}
\Psi((z,[\gamma]),w',\lambda)e^{\lambda(w'-w)} = B(w,\lambda) \Psi ((z,[\gamma]),w,\lambda) .
\end{align}
From (\ref{QPAPz}) and (\ref{PePBz}),
\begin{align*}
A(w',\lambda)& =\displaystyle \frac{Q_{(z,[\gamma])} \Psi((z,[\gamma]),w',\lambda)}{\Psi((z,[\gamma]),w',\lambda)}  = \frac{Q_{(z,[\gamma])} (e^{\lambda(w-w')} B(w,\lambda) \Psi((z,[\gamma]),w,\lambda) )}{B(w,\lambda) \Psi((z,[\gamma]),w,\lambda)  e^{\lambda(w-w')}}\\
& =\displaystyle \frac{Q_{(z,[\gamma])} \Psi((z,[\gamma]),w,\lambda)}{\Psi ((z,[\gamma]),w,\lambda)}  =A(w,\lambda).
\end{align*}
This shows that $A(w,\lambda)$ does not depend on the variable $w$. 
So, we set $A(\lambda)=A(w,\lambda)$.
Hence, the relation (\ref{QPAPz}) becomes to be
 \begin{align}\label{QeAxez}
Q_{(z,[\gamma])} \left(\left(\displaystyle \sum_{s=0}^\infty \xi_s((z,[\gamma]),w) \lambda^{-s} \right) e^{\lambda (z-w)}\right) = \left( \displaystyle \sum_{s=\alpha}^\infty A_s \lambda^{-s}\right)  \left(\displaystyle \sum_{s=0}^\infty \xi_s((z,[\gamma]),w ) \lambda^{-s} \right) e^{\lambda (z-w)}.
\end{align}
Since $Q_{(z,[\gamma])}$ is a differential operator of rank $M$,
a non-zero term which contains $\lambda^M$ appears in the left hand side of (\ref{QeAxez}) as the higher term in $\lambda$.
Therefore, considering the right hand side of (\ref{QeAxez}),
the series of $A(\lambda)$ must be in the form $A(\lambda)=\displaystyle \sum_{s=-M}^\infty A_s \lambda^{-s}$.

Conversely, we assume that the relation (\ref{QPAPz}) holds.
Then, we have
 $P_{z} Q_{(z,[\gamma])} \Psi((z,[\gamma]),w,\lambda)= P_{z} A(\lambda) \Psi((z,[\gamma]),w,\lambda)=\lambda^N A(\lambda) \Psi((z,[\gamma]),w,\lambda)$.
 This is clearly equal to $Q_{(z,[\gamma])} P_{z} \Psi((z,[\gamma]),w,\lambda)$. 
  Therefore, we have
 \begin{align}\label{PQCz}
 [P_{z},Q_{(z,[\gamma])}]\Psi((z,[\gamma]),w,\lambda)=0.
 \end{align}
Here, the relation (\ref{PQCz}) means that the ordinary differential equation $[P_{z},Q_{(z,[\gamma])}] u =0$ has solutions $\{ \Psi ((z,[\gamma]),w,\lambda) \}_\lambda$ parametrized by $\lambda$.
Since $\Psi$ is given by the form of (\ref{Psiz}),
it follows that the differential operator $[P_{z},Q_{(z,[\gamma])}] $ must be $0$.
\end{proof}

\begin{prop}\label{PQQCom}
Let $P_z$ be the differential operator of (\ref{Pz}).
Let $Q^{(1)}_{z,[\gamma]}$ and $Q^{(2)}_{z,[\gamma]}$ be the differential operator given by the form (\ref{Qz}).
If $P_z$ commutes with both $Q^{(1)}_{z,[\gamma]}$ and $Q^{(2)}_{z,[\gamma]}$,
then $Q^{(1)}_{z,[\gamma]}$ commutes with $Q^{(2)}_{z,[\gamma]}$.
\end{prop}

\begin{proof}
By the assumption and Proposition \ref{PAsz},
there exist series $A^{(1)}(\lambda)$ and $A^{(2)}(\lambda)$
in $\lambda$
such that 
$Q^{(j)}_{(z,[\gamma])}\Psi((z,[\gamma]),w,\lambda)=A^{(j)}(\lambda) \Psi((z,[\gamma]),w,\lambda)$
$(j=1,2)$ for $\Psi$ of (\ref{Psiz}).
So, we have
$$
[Q^{(1)}_{(z,[\gamma])},Q^{(2)}_{(z,[\gamma])}]\Psi((z,[\gamma]),w,\lambda) =(A^{(1)}(\lambda)A^{(2)}(\lambda)-A^{(2)}(\lambda)A^{(1)}(\lambda))\Psi((z,[\gamma]),w,\lambda)=0.
$$
As in the end of the proof of Proposition \ref{PAsz}, 
we have
$[Q^{(1)}_{(z,[\gamma])},Q^{(2)}_{(z,[\gamma])}]=0.$
\end{proof}

For the differential operator $P_{z}$ of (\ref{Pz}),
let $\mathcal{L}(P_{z},X)$ be the space of solutions of the differential equation
$P_{z} u = Xu$.
Suppose $Q_{(z,[\gamma])}$ of (\ref{Qz}) is a differential operator which commutes with $P_{z}$.
Then,
$Q_{(z,[\gamma])}$ defines a linear operator $\mathcal{Q}_{[\gamma],X}$ on the vector space $\mathcal{L}(P_{z},X).$

\subsection{Monodromy}

Let us take two arcs $\gamma$ and $\gamma'$ from $w$ to $z$ in $\mathbb{C}-\mathcal{N}.$
Setting $\delta=\gamma^{-1} \cdot \gamma' ,$
$[\delta]$ gives an element of the fundamental group $\pi_1 (\mathbb{C}-\mathcal{N}).$
For $\lambda$ such that $\lambda^N=X$,
since the coefficients of $P_z$ of (\ref{Pz}) are single-valued,
each $\Psi((z,[\gamma]),w,\lambda)$ and $\Psi((z,[\gamma']),w,\lambda)$ are solutions of the differential equation (\ref{DiffXz}) for generic $X$.
Based on Lemma \ref{LemIndp},
setting the vector 
\begin{align}\label{YokoPsi}
\Psi_v((z,[\gamma]),w,\lambda)=(\Psi((z,[\gamma]),w,\lambda_1),\cdots,\Psi((z,[\gamma]),w,\lambda_N)),
\end{align}
there exists a matrix $M({[\delta],w,\lambda)}\in GL(N,\mathbb{C})$ such that
\begin{align} \label{MonodromyM}
\Psi_v((z,[\gamma']),w,\lambda) = \Psi_v((z,[\gamma]),w,\lambda) M({[\delta]},w,\lambda).
\end{align}
The matrix $M([\delta],w,\lambda)$ is called the monodromy matrix of $[\delta]\in \pi_1(\mathbb{C}-\mathcal{N})$ for the system $\Psi_v ((z,[\gamma]),w,\lambda)$ of (\ref{YokoPsi}).
We note that
\begin{align}\label{MonoRepn}
r:\pi_1(\mathbb{C}-\mathcal{N}) \rightarrow GL(N,\mathbb{C}) \quad \text {given by} \quad  [\delta] \mapsto M({[\delta]},w,\lambda)
\end{align}
is a homomorphism of groups.

Let $\gamma, \gamma'$ and $\delta$ be as above.
Suppose $Q_{(z,[\gamma])}$ commutes with $P_z$ of (\ref{Pz}).
By considering the analytic continuation along the closed arc $\delta$,
$Q_{ (z,[\gamma'])}$ also commutes with $P_z$.
For the linear operator $\mathcal{Q}_{[\gamma],X}$ on $\mathcal{L}(P_z,X)$, set
\begin{align}\label{[delta]}
[\delta]^* (\mathcal{Q}_{[\gamma],X}) = \mathcal{Q}_{[\gamma'],X}.
\end{align}

\begin{thm} \label{ThmPQCommute}
Let $\gamma,\gamma',\delta,  \mathcal{Q}_{[\gamma],X}$ be as above.

(1) The set of the eigenvalues of the linear operator $\mathcal{Q}_{[\gamma],X}$ coincides with that of the linear operator $[\delta]^*(\mathcal{Q}_{[\gamma],X})$ for generic $X$.

(2) 
There exists $N_0\in\mathbb{Z}_{>0}$ such that $([\delta]^*)^{N_0} ( \mathcal{Q}_{[\gamma],X}) = \mathcal{Q}_{[\gamma],X}$
for any $[\delta]\in \pi_1(\mathbb{C}-\mathcal{N})$ and generic $X$.

\end{thm}

\begin{proof}
Set
$
Q_{(z,[\gamma])} \Psi_v((z,[\gamma]),w,\lambda) = (Q_{(z,[\gamma])} \Psi((z,[\gamma]),w,\lambda_1),\cdots,Q_{([z],\gamma)} \Psi((z,[\gamma]),w,\lambda_N)).
$
According to Proposition \ref{PAsz} together with  the notation (\ref{YokoPsi}), 
we have
$$
Q_{(z,[\gamma] )} \Psi_v((z,[\gamma]),w,\lambda)=\Psi_v((z,[\gamma]),w,\lambda) 
\begin{pmatrix} A(\lambda_1) &  & 0 \\ & \cdots & \\ 0 & & A(\lambda_N)\end{pmatrix}.
$$
For $[\delta]\in \pi_1 (\mathbb{C}-\mathcal{N})$, 
using the  monodromy matrix 
$M({[\delta]},w,\lambda)$ 
of (\ref{MonodromyM}),
we have
\begin{align}\label{QQQM}
\notag
Q_{(z,[\gamma'])} \Psi_v((z,[\gamma]),w,\lambda)
& = Q_{(z,[\gamma'])} \Psi_v((z,[\gamma']),w,\lambda)M({[\delta]},w,\lambda)^{-1} \\
&\notag =  \Psi_v((z,[\gamma']),w,\lambda) 
\begin{pmatrix} A(\lambda_1) &  & 0 \\ & \cdots & \\ 0 & & A(\lambda_N) \end{pmatrix}
M({[\delta]},w,\lambda)^{-1}\\
&=    \Psi_v((z,[\gamma]),w,\lambda) M({[\delta]},w,\lambda)
\begin{pmatrix} A(\lambda_1) &  & 0 \\ & \cdots & \\ 0 & & A(\lambda_N) \end{pmatrix}
M({[\delta]},w,\lambda)^{-1}.
\end{align}
Here, 
we used the fact that 
$Q_{(z,[\gamma'])} \Psi ((z,[\gamma']),w,\lambda_j) = A(\lambda_j ) \Psi ((z,[\gamma']),w,\lambda_j)$ due to Proposition \ref{PAsz}.

On the other hand,
since $Q_{(z,[\gamma'])}$ commutes with $P_z$,
we can directly apply  Proposition \ref{PAsz} to  $Q_{(z,[\gamma'])}$.
Then, there exist $A'(\lambda_1),\cdots, A'(\lambda_N)$ such that
\begin{align}\label{QPPA}
Q_{(z,[\gamma'])} \Psi_v((z,[\gamma]),w,X)
=\Psi_v((z,[\gamma]),w,\lambda)
 \begin{pmatrix} A'(\lambda_1) &  & 0 \\ & \cdots & \\ 0 & & A'(\lambda_N) \end{pmatrix}.
\end{align} 

So, from (\ref{QQQM}) and (\ref{QPPA}),
we have
\begin{align*}
\begin{pmatrix} A'(\lambda_1) &  & 0 \\ & \cdots & \\ 0 & & A'(\lambda_N) \end{pmatrix}
=
M({[\delta]},w,\lambda)
\begin{pmatrix} A(\lambda_1) &  & 0 \\ & \cdots & \\ 0 & & A(\lambda_N) \end{pmatrix}
M({[\delta]},w,\lambda)^{-1}.
\end{align*}
This implies that the set of eigenvalues $\{A(\lambda_1),\cdots,A(\lambda_N)\}$ for $Q_{[\gamma],X}$ coincides with that of eigenvalues $\{A'(\lambda_1),\cdots,A'(\lambda_N)\}$ for $\mathcal{Q}_{[\gamma'],X}$.

(2)
From the above (1),
the correspondence $[\delta]$ of (\ref{[delta]}) induces  a permutation of the eigenvalues.
Therefore, setting $N_0 = N!$,
we have
$$
(([\delta]^*)^{N_0} Q_{(z,[\gamma])}) \Psi_v((z,[\gamma]),w,\lambda) =   \Psi_v ((z,[\gamma]),w,\lambda) \begin{pmatrix} A(\lambda_1) &  & 0 \\ & \cdots & \\ 0 & & A(\lambda_N) \end{pmatrix} = Q_{(z,[\gamma])} \Psi_v ((z,[\gamma]),w,\lambda).
$$
So, for $j\in \{1,\cdots,N\}$,
we have
$
(([\delta]^*)^{N_0} (Q_{(z,[\gamma])})-Q_{(z,[\gamma])})  \Psi((z,[\gamma]),w,\lambda_j)=0.
$
Therefore,
by a similar argument to the end of the proof of Proposition \ref{PAsz},
we obtain $([\delta]^*)^{N_0} (Q_{(z,[\gamma])})=Q_{(z,[\gamma])}.$
\end{proof}

\begin{cor}
If $\mathcal{N}$ is a finite set of $\mathbb{C}$,
the coefficients $b_k(z,[\gamma])$ ($k=0,\cdots,M$) of $Q_{(z,[\gamma])}$ of (\ref{Qz}) 
are at most
algebraic functions of $z$. 
\end{cor}

\begin{proof}
By Theorem \ref{ThmPQCommute}
and
the assumption of the corollary,
the image of 
the correspondence
$\pi_1 (\mathbb{C} -\mathcal{N}) \ni [\delta] \mapsto [\delta]^*$ of (\ref{[delta]})
is  finite.
This implies the assertion.
\end{proof}

\subsection{The algebraic curve $\mathcal{R}$}

For generic $X\in \mathbb{C}$,
we have the distinct $N$ values $\lambda_1,\cdots, \lambda_N$ such that $\lambda_j^N=X$ $(j=1,\cdots, N)$.
From Lemma \ref{LemIndp},
$\{\Psi((z,[\gamma]),w,\lambda_j) | j\in\{ 1,\cdots,N\} \}$ gives a system of basis of $\mathcal{L}(P_{z},X)$.
From Proposition  \ref{PAsz},
$\Psi((z,[\gamma]),w,\lambda_j)$ gives an eigenfunction with the eigenvalue $A(\lambda_j)$ of $\mathcal{Q}_{[\gamma],X}$.
Hence,
the characteristic polynomial of $\mathcal{Q}_{[\gamma],X}$ on $\mathcal{L}(P_{z},X)$ is given by
\begin{align}\label{eigenz}
\prod_{j=1}^N (Y-A(\lambda_j)).
\end{align}

\begin{lem}\label{LemCj0z}
Let $\{C_l((z,[\gamma]),w,X) \}_{l=0,1,\cdots,N-1}$
be a system of basis of the vector space $\mathcal{L}(P_{z},X)$ satisfying
\begin{align}\label{Cjz}
\frac{\partial^r}{\partial z^r} C_l((z,[\gamma]),w,X) \Big| _{(z,[\gamma])=(w,[id])}=\delta_{l,r}.
\end{align}

(1) For fixed $w\in \mathbb{C}-\mathcal{N}$ and $(z,[\gamma])\in \mathfrak{X}$,
the correspondence $X\mapsto C_l((z,[\gamma]),w,X)$ gives a holomorphic function on $\mathbb{C}=\mathbb{P}^1(\mathbb{C})-\{\infty\}$. 

(2)
For $Q_{(z,[\gamma])}$ of (\ref{Qz}),
the components of the representation matrix of the linear operator $\mathcal{Q}_{[\gamma],X}$ for the system of basis $\{C_l((z,[\gamma]),w,X)\}_{l=0,1,\cdots,N-1}$ are given by polynomials in $X$ and special values of $a_j(z)$ $(j=0,\cdots,N)$ and  $b_j((z,[\gamma]),w)$ $(k=0,\cdots,M)$.
\end{lem}

\begin{proof}
(1)  The solutions $C_l((z,[\gamma]),w,X)$ are given by solving the initial value problem for the differential equation (\ref{DiffXz}).
Hence, the correspondence $X\mapsto C_l((z,[\gamma]),w,X)$ is holomorphic.

(2) Since $C_l((z,[\gamma]),w,X)$ $(l=0,\cdots,N-1)$ are solutions of the equation (\ref{DiffXz}), 
we obtain
\begin{align}\label{indNz}
\frac{\partial^N}{\partial z^N} C_l((z,[\gamma]),w,X) = X C_l((z,[\gamma]),w,X)-\displaystyle \sum_{k=0}^{N-1}  a_k(z)\frac{\partial^k}{\partial z^k} C_l((z,[\gamma]),w,X).
\end{align}
By the way, 
since $Q_{(z,[\gamma])} C_l((z,[\gamma]),w,X)\in \mathcal{L}(P_{z},X),$
there exists constants $c_{l,m} (w,X)$ $(l,m\in\{0,\cdots,N-1\})$ 
for $z$
such that
$Q_{(z,[\gamma])} C_l((z,[\gamma]),w,X) = \displaystyle \sum_{m=0}^{N-1} c_{l,m}  C_m((z,[\gamma]),w,X)$.
Due to (\ref{Cjz}), we have
\begin{align}\label{restrCz}
c_{l,m}(w,X)=\frac{\partial^m}{\partial z^m}Q_{(z,[\gamma])}  C_l((z,[\gamma]),w,X) \Big|_{(z,[\gamma])=(w,[id])}.
\end{align}
Using the relation (\ref{indNz}),
we can see that
$\frac{\partial^r}{\partial z^r} C_l((z,[\gamma]),w,X) \Big|_{(z,[\gamma])=(w,[id])}$ are given by a polynomial in $X$ and 
the special values $a_k(w)$ 
for any $r\in \mathbb{Z}$.
So, according to (\ref{restrCz}),
we can see that $c_{l,m}(w,X)$ are given by  polynomials in $X$ and special values of $a_j(z)$
and $b_k((z,[\gamma]),w)$ .
\end{proof}

We note that $\lambda_j$ $(j=1,\cdots, N)$ 
are distinct solutions of the algebraic equation $\lambda^N=X$.
From (\ref{QPAPz}),
$A(\Omega,\lambda_j)$ is a Laurent series in $\lambda_j^{-1}$.
Since the right hand side of (\ref{eigenz}) is symmetric series in $\lambda_j^{-1}$ $(j=1,\cdots, N)$, the right hand side of (\ref{eigenz}) gives a Laurent series in $X^{-1}$.
Moreover, we have the following.

\begin{cor}\label{LemCj}
The characteristic polynomial (\ref{eigenz}) defines a polynomial in $X $. 
\end{cor}

\begin{proof}
We have a representation matrix of $\mathcal{Q}_{[\gamma],X}$ 
whose components are polynomial in $X$ from the above lemma. 
Therefore, its characteristic polynomial is given by a polynomial in $X$.
\end{proof}

In the following,
let $F (X,Y)$ be
the polynomial (\ref{eigenz}) in the variables $X$ and $Y$.

\begin{thm}\label{ThmRiemannSurfacez}
The differential operators $P_{z}$ of (\ref{Pz}) and $Q_{(z,[\gamma])}$ of (\ref{Qz})
satisfy
$F (P_{z},Q_{(z,[\gamma])})=0.$
\end{thm}

\begin{proof}
For generic $X \in \mathbb{C}$,
letting $\lambda$ be a solution of the equation of $\lambda^N=X$, we have
\begin{align*}
F (P_{z},Q_{(z,[\gamma])})\Psi((z,[\gamma]),w,\lambda)=F (X, Q_{(z,[\gamma])})\Psi((z,[\gamma]),w,\lambda)=0.
\end{align*}
Here, the last equality is due to the Hamilton-Cayley theorem.
Then, the ordinary differential equation $F(P_{z},Q_{(z,[\gamma])}) u=0$ 
has a family $\{\Psi((z,[\gamma]),w,\lambda)\}_\lambda$ of solutions with the parameter $\lambda$.
By a similar argument to the end of the proof of Proposition \ref{PAsz},
the operator $F (P_{z},Q_{(z,[\gamma])})$ is equal to $0$.
\end{proof}

The equation 
$
F (X,Y)=0
$
defines an algebraic curve $\mathcal{R}$.
This curve should be in the form
\begin{align}\label{AlgCz}
\mathcal{R} :\displaystyle \sum_{j,k} f_{j,k} X^j Y^k=0.
\end{align}

Let $\mathcal{R}$ be the algebraic curve in Theorem \ref{ThmRiemannSurfacez}.
Let $\pi:\mathcal{R} \rightarrow \mathbb{P}^1(\mathbb{C})$ be the projection given by $(X,Y) \mapsto X$.
Let $p_\infty$ be the point of $\mathcal{R}$ corresponding to $X=\infty\in \mathbb{P}^1(\mathbb{C}).$
Then, $p_\infty$ is a ramification point  of the mapping $\pi.$
We note that $X_1=\frac{1}{X}$ gives a complex coordinate around $p_\infty \in \mathcal{R}.$

By the procedure of the algebraic curve $\mathcal{R}$ and the covering $\pi:\mathcal{R}\rightarrow \mathbb{P}^1(\mathbb{C}),$
Proposition \ref{PAsz}
and
Theorem \ref{ThmPQCommute} (1)
imply that 
any $[\delta]\in \pi_1 (\mathbb{C}-\mathcal{N})$
induces the correspondence
$
\sigma_{[\delta]}: \mathcal{R} \rightarrow \mathcal{R}
$
given by
\begin{align}\label{sigmaD}
p_j = (X,A(\lambda_j)) \mapsto \sigma_{[\delta]}(p_j)= p_k=(X,A(\lambda_k)),
\end{align}
when 
\begin{align}\label{sigmaC}
([\delta]^*(\mathcal{Q}_{[\gamma],X})) \Psi((z,[\gamma]),w,\lambda_j)= A(\lambda_k) \Psi ((z,[\gamma]),w,\lambda_j).
\end{align}
Hence, 
letting ${\rm Aut}(\pi)$ be the group of transformations for the covering $\pi$,
we have the homomorphism
$$
\pi_1 (\mathbb{C}-\mathcal{N}) \rightarrow {\rm Aut}(\pi)
$$
 of groups given by $[\delta]\mapsto \sigma_{[\delta]}$.

\begin{thm}\label{ThmSingleSigma}
(1)
All coefficients of the operator $Q_{(z,[\gamma])}$ are single-valued on $\mathbb{C} - \mathcal{N}$
if and only if
$\sigma_{[\delta]} =id$ for every $[\delta]\in \pi_1 (\mathbb{C}-\mathcal{N})$.

(2)
For $\lambda_j$ $(j=1,\cdots, N)$ satisfying $\lambda_j^N=X$,
assume $A(\lambda_1) , \cdots, A(\lambda_N)$ are distinct for generic $X$. 
Then,
$\sigma_{[\delta]} = id$
if and only if
\begin{align}\label{PmuP}
\Psi ((z,[\gamma']),w,\lambda_j) = \mu_j \Psi((z,[\gamma]),w,\lambda_j) \quad \quad (j=1,\cdots,N),
\end{align}
where $\mu_j$ is a constant  function of $z$.
\end{thm}

\begin{proof}
(1) If all coefficients of $Q_{(z,[\gamma])}$ are single-valued,
we have 
$[\delta]^*(\mathcal{Q}_{[\gamma],X}) = \mathcal{Q}_{[\gamma],X}$
for any $[\delta]\in \pi_1(\mathbb{C}-\mathcal{N})$.
Then,
by (\ref{sigmaD}) and (\ref{sigmaC}),
we have $\sigma_{[\delta]}=id.$

Conversely,
if $\sigma_{[\delta]}=id$ for any $[\delta]\in \pi_1(\mathbb{C}-\mathcal{N})$,
from (\ref{sigmaD}) and (\ref{sigmaC}),
we have 
$$
[\delta]^*(\mathcal{Q}_{[\gamma],X}) \Psi ((z,[\gamma]),w,\lambda_j)=A(\lambda_j) \Psi ((z,[\gamma]),w,\lambda_j)=\mathcal{Q}_{[\gamma],X} \Psi ((z,[\gamma]),w,\lambda_j)
$$
for generic $X$.
So, by a similar argument to the proof of Proposition \ref{PAsz},
we have $Q_{(z,[\gamma'])}=Q_{(z,[\gamma])},$ where $\gamma' = \gamma \cdot \delta$.
Hence, the assertion holds.

(2) By the assumption,
$\Psi((z,[\gamma]),w,\lambda_j)$ spans the $1$-dimensional eigenspace for the eigenvalue $A(\lambda_j) $ of $\mathcal{Q}_{[\gamma],X}$.
Set $\gamma' = \gamma \cdot \delta.$
If $\sigma_{[\delta]}=id$,
from (1),
we have $([\delta^{-1}]^*) \mathcal{Q}_{[\gamma'],X} = \mathcal{Q}_{[\gamma'],X}$ for any $[\delta]\in \pi_1(\mathbb{C}-\mathcal{N})$.
This implies that
$$
Q_{z,[\gamma]} \Psi((z,[\gamma']),w,\lambda_j) = Q_{z,[\gamma']} \Psi((z,[\gamma']),w,\lambda_j)  = A(\lambda_j) \Psi((z,[\gamma']),w,\lambda_j) 
$$
for generic $X$, where $\lambda_j^N=X$.
Here, we used Proposition \ref{PAsz}.
Therefore, $ \Psi((z,[\gamma']),w,\lambda_j) $ is an eigenfunction for the eigenvalue $A(\lambda_j)$.
So, $\Psi((z,[\gamma']),w,\lambda_j) \in \langle \Psi((z,[\gamma]),w,\lambda_j) \rangle_\mathbb{C}$ holds.

Conversely,
if we have (\ref{PmuP}),
then,
due to Proposition \ref{PAsz}, 
\begin{align*}
&Q_{(z,[\gamma'])} \Psi((z,[\gamma]),w,\lambda_j) =  \mu_j^{-1} Q_{(z,[\gamma'])}  \Psi((z,[\gamma']),w,\lambda_j) \\
&= \mu_j^{-1} 
A(\lambda_j)  \Psi((z,[\gamma']),w,\lambda_j) = Q_{(z,[\gamma])} \Psi((z,[\gamma]),w,\lambda_j) 
\end{align*}
holds for generic $X$.
Hence,
as in  (1),
we have $\sigma_{[\delta]}=id.$
\end{proof}

\subsection{The eigenfunction $\psi$}

In this subsection,
we use the same notation which we use in the previous subsection.
Moreover,
we suppose that 
\begin{align}\label{assumez}
\text{ there exists } s \hspace{1mm} ( s \geq -M),
\text{ where }  N \text{ and } s \text{ are coprime },  
\text{ such that } A_s \not = 0
\end{align}
for $\{A_s\}$ of (\ref{star}).
Then,
the operator $\mathcal{Q}_{[\gamma],X}$ on $\mathcal{L}(P_z,X)$ has $N$ distinct eigenvalues $A(\lambda_j)$ $(j=1,\cdots,N)$ in the sense of Proposition \ref{PAsz}.
Hence, the eigenspace for the eigenvalue $A_s(\lambda_j)$ is $1$-dimensional.

Since $\mathfrak{X}$ is simply connected,
for a general $X\in \mathbb{C}$ and $p\in \pi^{-1}(X)\subset \mathcal{R}$,
we can take  the unique eigenfunction on $\mathfrak{X}$:
\begin{align}\label{psiz}
\psi((z,[\gamma]),w,p)
=\displaystyle \sum_{l=0}^{N-1} h_l(w,p) C_l((z,[\gamma]),w,X),
\end{align}
where $h_0(w,p)\equiv 1$.
Here, $C_l((z,[\gamma]),w,X)$ $(l=0,\cdots,N-1)$ are given in Lemma \ref{LemCj0z}
and
$h_l(w,p)$ does not depend on $z$.

\begin{lem} \label{Lempsiz}
Let $\psi((z,[\gamma]),w,p)$ be the function of (\ref{psiz}).

(1) For fixed  $w\in \mathbb{C}-\mathcal{N}$,
$p\mapsto h_l(w,p)$ gives a meromorphic function on $\mathcal{R} - \{p_\infty\}$.

(2) For fixed $w \in \mathbb{C}-\mathcal{N}$,
the poles of $\mathcal{R} - \{p_\infty\} \ni p \mapsto \psi((z,[\gamma]),w, p)\in \mathbb{P}^1 (\mathbb{C})$ do not depend on $(z,[\gamma]) \in \mathfrak{X}.$  

(3) Let $U_\infty \subset \mathbb{P}^1(\mathbb{C})$ be a sufficiently small neighborhood of $X=\infty.$
Let $V\subset \mathbb{C}-\mathcal{N}$ be a sufficiently small and simply connected neighborhood of $w$.
If $\pi(p) \in U_{\infty} - \{\infty\}$, $z\in V$ and $\gamma \subset V$,
then
$p\mapsto \psi((z,[\gamma]),w,p)$
is analytic and
 has an exponential singularity at $p=p_\infty$.

\end{lem}

\begin{proof}
(1) We had the representation matrix
$c (w,X)=(c_{jk}(w,X))$ of the linear operator 
$\mathcal{Q}_{[\gamma],X}$ on $\mathcal{L}(P_z,X)$
for the system of basis $\{C_l((z,[\gamma]),w,X)\}_{l=0,\cdots, N-1}$ of (\ref{Cjz}).
Here, by  Lemma \ref{LemCj0z} (2),
$c_{l,m}(w,X)$ are given by polynomials in $X$.
Let $p\in\mathcal{R}$ be a point corresponding to $X\in \mathbb{C}$ and the eigenvalue $Y$.
We can obtain $h_l(w,p) $ of (\ref{psiz}) by solving the linear equation
$$
c(w,X)
\begin{pmatrix}
h_0(w,p) \\ h_1(w,p) \\ \cdots \\ h_{N-1}(w,p)
\end{pmatrix}
=
Y
\begin{pmatrix}
h_0(w,p) \\ h_1(w,p) \\ \cdots \\ h_{N-1}(w,p)
\end{pmatrix},
$$
where $h_0(w,p)\equiv 1$.
This implies that $h_l(w,p)$ $(l=1,\cdots,N-1)$ are given by rational functions of $X$ and $Y$.
Therefore, $p \mapsto h_l (w,p)$ is meromorphic on $\mathcal{R}$.

(2) 
From Lemma \ref{LemCj0z} (1) and the expression  (\ref{psiz}) of $\psi$,
the poles of  $\mathcal{R}- \{p_\infty\} \ni p \mapsto \psi((z,[\gamma]),w,p)\in \mathbb{P}^1 (\mathbb{C})$
 are coming only from the poles of $p \mapsto h_l(w,p)$ $(l=1,\cdots,N-1 )$.
 These poles do not depend on $(z,[\gamma])$.
 
 (3)
 From the procedure of the Riemann surface $\mathcal{R}$,
 we can take sufficiently small neighborhood $U_\infty$ such that
 the set $\pi^{-1}(X)$ consists $N $ distinct points for any $ X \in U_\infty -\{p_\infty\}$.
 Then,
 $p\in \mathcal{R} -\{p_\infty\}$ such that $\pi(p)=X \in U_\infty$
 corresponds to
$(X,Y)=(X,A(\lambda_j))$ for $j=1,\cdots, N$,
where $\lambda_j^N=X$. 
Then, $\psi ((z,[\gamma]),w,p)$ corresponds to $\Psi((z,[\gamma]),w,\lambda_j)$ of (\ref{Psiz}) and (\ref{psiz}).
So, from (\ref{Psiz}),
\begin{align}\label{h_1z}
h_1(w,p)=\frac{\partial}{\partial z} \Psi((z,[\gamma]),w,\lambda_j)\Big|_{(z,[\gamma])=(w,[id])}=\lambda_j(1+O(\lambda_j^{-1})).
\end{align}

By the way,
we take a sufficiently small and simply connected neighborhood $V\subset \mathbb{C}-\mathcal{N}$ of $w$.
Let $x\in V$.
We have 
the logarithmic derivative of $\psi$ at $w$:
$$
\frac{\partial}{\partial z} \log \psi((z,[\gamma]),w,p)\Big|_{(z,[\gamma])=(w,[id])}  =\frac{\frac{\partial}{\partial z} \psi((z,[\gamma]),w,p) }{\psi((z,[\gamma]),w,p)}\Big|_{(z,[\gamma])=(w,[id])} = h_1(w,P).
$$
By changing the base point $w$, which defines  the universal covering $\mathfrak{X}$, to a  point $x$ of the simply connected neighborhood $V$,
we can regard $x\mapsto h_1 (x,p)$ as a single-valued holomorphic function on $V$.
So, if $z\in V$ and $\gamma \subset V$, then we locally have the expression
\begin{align}\label{psih_1z}
\psi((z,[\gamma]),w,p)={\rm exp} \Big( \int_\gamma h_1(x,p) dx \Big).
\end{align} 
From (\ref{h_1z}) and (\ref{psih_1z}), we have the assertion of (3).
\end{proof}

\begin{rem}
The expression (\ref{psih_1z}) is valid only for sufficiently close $(x,[\gamma])$ to $(w,[id])$,
because we used the change of the base point from $w$ to $x$.
We note that $h_1(x,p) $ depends on the choice of the base point.
Generically, $x\mapsto h_1(x,p)$ can be globally multivalued and the expression (\ref{psih_1z}) does not holds.
\end{rem}

For $X\in \mathbb{C}$, $\pi^{-1}(X)=p_1+\cdots +p_N $ gives a divisor on $\mathcal{R}$.
We set
\begin{align}\label{Gz}
G((z,[\gamma]),w,X)=
\left({\rm det} 
\begin{pmatrix}
\psi((z,[\gamma]),w,p_1) & \cdots & \psi((z,[\gamma]),w,p_N)\\
\frac{\partial}{\partial z}\psi((z,[\gamma]),w,p_1) & \cdots & \frac{\partial}{\partial z}\psi((z,[\gamma]),w,p_N)\\
\cdots & \cdots & \cdots \\
\frac{\partial^{N-1}}{\partial z^{N-1}}\psi((z,[\gamma]),w,p_1) & \cdots & \frac{\partial^{N-1}}{\partial z^{N-1}}\psi((z,[\gamma]),w,p_N)\\
\end{pmatrix}
\right)^2.
\end{align}
for fixed $(z,[\gamma])$ and $w$.
We note that 
$X\mapsto G((z,[\gamma]),w,X)$ is well-defined on $X\in \mathbb{C}$.

\begin{lem}\label{LemPole}
For  $(z,[\gamma])\in \mathfrak{X}$ and $w\in \mathbb{C}-\mathcal{N}$,
any poles of the function 
$\mathcal{R}-\{\infty\} \ni p \mapsto \psi((z,[\gamma]),w,p) \in \mathbb{P}^1( \mathbb{C})$
analytically depend on the base point $w$.
They are not independent of $w$.
\end{lem}

\begin{proof}
For  a fixed  base point $w$,
if  $q$ is a pole of $p \mapsto \psi((z,[\gamma]),w,p)$,
by Lemma \ref{Lempsiz} (2),
it holds that $\psi((z,[\gamma]),w,q)=\infty$ for any $(z,[\gamma])\in \mathfrak{X}$.
So, if we can take $q$  which is independent of $w$,
we have $\psi((z,[\gamma]),w,q)=\infty$ for any $(z,[\gamma])\in \mathfrak{X}$ and $w\in \mathbb{C}-\mathcal{N}.$
This is a contradiction, 
because
we have $\psi((w,[id]),w,q)=1$  by the definition (\ref{psiz}) of $\psi$.
So, any pole of $p \mapsto \psi((z,[\gamma]),w,p)$ is not independent of $w$.
Moreover,
from the proof of Lemma \ref{Lempsiz} (1),
such poles are coming from the zeros of the common denominator of  $h_1(w,p),\cdots,h_{N-1}(w,p)$.
We note that the common denominator is given by a polynomial in $X$ and $Y$ analytically  parametrized by $w$.
So, poles  analytically depend on $w$.
\end{proof}

\begin{lem}\label{LemGz}
Take    $w\in \mathbb{C}-\mathcal{N}$ and $(z,[\gamma])\in\mathfrak{X}$.
Then, the correspondence $X\mapsto G((z,[\gamma]),w,X)$ of (\ref{Gz}) gives 
a rational function of $X$.
Moreover,
this rational function
has a pole at $X=\infty$ of degree $N-1$.
\end{lem}

\begin{proof}
Due to Lemma \ref{Lempsiz} (1) and the properties of determinant of (\ref{Gz}),
the correspondence
$X\mapsto G((z,[\gamma]),w,X)$
defines a meromorphic function on $\mathbb{C}=\mathbb{P}^1(\mathbb{C})-\{\infty\}$.
Now, take a sufficiently small and simply connected neighborhood $V \subset \mathbb{C}-\mathcal{N}$ of $w$.
Although $p\mapsto \psi ((z,[\gamma]),w,p)$ for $z\in V$ and $\gamma\subset V$ has an exponential singularity at $p_\infty$ (Lemma \ref{Lempsiz} (3)),
 we will see that $X\mapsto G((z,[\gamma]),w,X)$ is analytic around $X=\infty$.
Taking a sufficiently small neighborhood $U_\infty$, 
$\psi$ of (\ref{psiz}) is given by $\Psi$ of (\ref{Psiz})
and holomorphic on $U_\infty-\{p_\infty\}$.
Then,
considering the properties of the determinant of (\ref{Gz}),
and the fact that $e^{\lambda_1(z-w)}\cdots e^{\lambda_N(z-w)}=1$,
we can see that $G((z,[\gamma]),w,X)$ has the form
\begin{align}\label{GOz}
\left( {\rm det} \begin{pmatrix} 
1+O(\lambda_1^{-1}) & \cdots & 1+O(\lambda_N^{-1})\\
\lambda_1(1+O(\lambda_1^{-1})) & \cdots & \lambda_N(1+O(\lambda_N^{-1}))\\
\cdots & \cdots & \cdots \\
\lambda_1^{N-1}(1+O(\lambda_1^{-1})) & \cdots & \lambda_N^{N-1}(1+O(\lambda_N^{-1}))\\
\end{pmatrix} \right)^2,
\end{align}
around $X=\infty.$
Then, (\ref{GOz}) is a symmetric series in $\lambda_1,\cdots, \lambda_N$
with the highest term of degree $2(0+1+\cdots + (N-1))=N(N-1)$.
Setting $X_1=\frac{1}{X}$,
$X_1$ gives a complex coordinate around $X=\infty$
and
(\ref{GOz}) gives a Laurent series in $X_1.$
Due to Lemma \ref{Lempsiz} (1) and the assumption
(\ref{assumez}), applying the Riemann extension theorem,
 (\ref{GOz}) is holomorphic at $X_1=0$ and has a zero  of degree $\frac{N(N-1)}{N}=N-1$
 for $z\in V$ and $\gamma \subset V$.
By the analytic continuation in terms of  $(z,[\gamma])\in \mathfrak{X}$,
we have the assertion. 
\end{proof}

\begin{thm}\label{Thmphipolez}
Assume the condition (\ref{assumez}).
Suppose the algebraic curve $\mathcal{R}$ given by the defining equation (\ref{AlgCz}) 
is non-singular and of genus $g$.
Then, 
for $w \in \mathbb{C}-\mathcal{N}$ and $(z,[\gamma])\in \mathfrak{X}$,
the function
\begin{align}\label{ppsiz}
\mathcal{R} - \{p_\infty\}\ni p \mapsto \psi((z,[\gamma]),w,p) \in \mathbb{P}^1(\mathbb{C})
\end{align}
has $g$ poles.
\end{thm}

\begin{proof}
Let $\kappa$ be the number of poles of the function of (\ref{ppsiz}).
Since $\psi$ is given by (\ref{psiz}), together with  Lemma \ref{Lempsiz} (2),
we can see that  the set of the poles of $p\mapsto \psi ((z,[\gamma]),w,p)$ corresponds to that of $p\mapsto \frac{\partial^r}{\partial z^r}\psi ((z,[\gamma]),w,p)$ $(r\geq1)$.
So, from the definition of the rational function $X\mapsto G((z,[\gamma]),w,X)$  of (\ref{Gz}), 
the number of the poles of the function
$$
\mathbb{P}^1(\mathbb{C})-\{\infty\} \ni X \mapsto G((z,[\gamma]),w,X) \in \mathbb{P}^1(\mathbb{C})
$$
is equal to  $2\kappa$.
Together with Lemma \ref{LemGz},
the number of poles of the function
\begin{align}\label{Gentz}
\mathbb{P}^1(\mathbb{C}) \ni X \mapsto G((z,[\gamma]),w,X) \in \mathbb{P}^1(\mathbb{C})
\end{align}
is equal to $2\kappa +N-1$.
Since (\ref{Gentz}) is a rational function of the variable $X$,
this function has $2 \kappa +N-1$ zeros on $\mathbb{P}^1(\mathbb{C})-\{\infty\}$.

On the other hand,
from Lemma \ref{LemPole} and the fact that the ramification points of $\pi $ are isolated points of $\mathcal{R}$,
for generic base point $w\in \mathbb{C}-\mathcal{N}$,
 all poles of $p\mapsto \psi((z,[\gamma]),w,p)$ 
are out of the set of the ramification points of $\pi$.
We fix such a base  point $w$.
From the definition (\ref{Gz}),
the function $X\mapsto G((z,[\gamma]),w,X)$ vanishes at $X$ $(\not = \infty)$
if and only if 
$X$ is a branch point of the covering $\pi.$
Letting $e_p$ be the ramification index of $\pi$ at $p\in \mathcal{R}$.
From the property of determinants of matrices,
the right hand side of (\ref{Gz}) has zeros of degree $2(0+1+(e_p-1))=e_p(e_p-1)$ 
of a coordinate around $p\in \mathcal{R}.$
So, at $X=\pi (p)$,
$X\mapsto G((z,[\gamma]),w,X)$
has  zeros of  degree $\frac{e_p(e_p-1)}{e_p}=e_p-1.$
Therefore, the degree of zeros of the function of (\ref{Gentz}) 
coincides with $\displaystyle \sum_{p\in \mathcal{R}-\{p_\infty\}} (e_p -1).$
So, together with Lemma \ref{LemGz}, 
\begin{align}\label{RHz}
\displaystyle \sum_{p\in \mathcal{R}} (e_p -1) = 2 \kappa +2 N -2.
\end{align}
By the way,  applying the Riemann-Hurwitz formula,
 we have
 \begin{align}\label{RiemannHurwitzz}
 \displaystyle \sum_{p\in \mathcal{R}} (e_p -1)= (2 g- 2)+N(2-0).
 \end{align} 
 By (\ref{RHz}) and (\ref{RiemannHurwitzz}), we have $\kappa=g$.
 Therefore, we have proved the assertion for generic $w$.
 Since
 the number of poles of $p\mapsto \psi((z,[\gamma]),w,p)$ on $\mathcal{R}$ is 
 analytically dependent on the variable $w$,
  this is a constant function of $w$.
Thus, for every $w$, the number of poles is equal to $g$.
\end{proof}

Next, 
we consider the case that the algebraic curve $\mathcal{R}$ of (\ref{AlgCz}) has singular points $\mathcal{S}$ $(\subset \mathcal{R})$.
We have a resolution of singularities $\sigma:\tilde{\mathcal{R}}\rightarrow \mathcal{R}$.
Here,
$\sigma$ is given by a composition 
$\tilde{\mathcal{R}}=\mathcal{R}_{l}\rightarrow \mathcal{R}_{l-1}\rightarrow \cdots \rightarrow \mathcal{R}_{0}=\mathcal{R}$
of  blowing ups $\sigma_\nu:\mathcal{R}_{\nu}\rightarrow \mathcal{R}_{\nu-1}$ 
for a singular point of multiplicity $m_\nu \in \mathbb{Z}_{>0}$
$(\nu=1,\cdots,\kappa)$.
We have an $N$ to $1$ covering $\pi\circ \sigma:\tilde{\mathcal{R}}\rightarrow \mathbb{P}^1(\mathbb{C}).$
By considering the divisor $(\pi\circ \sigma)^{-1}(X) $ for $X\in \mathbb{P}^1(\mathbb{C})-\{\infty\},$
we can define the function 
$X\mapsto G((z,[\gamma]),w,X)$, also.
By a similar argument of the proof of Theorem \ref{Thmphipolez}
and
considering properties of the blowing ups (for example, see \cite{G}),
$X\mapsto G((z,[\gamma]),w,X)$ has zeros, 
not only
at the branch points of $\pi\circ \sigma$,
but also
the images of $\mathcal{S}$ under $\pi$,
where the sum of  the orders of zeros is at most $\displaystyle \sum_{\nu=1}^l m_\nu (m_\nu -1).$
Applying the argument  of the proof of Theorem \ref{Thmphipolez}
to the non-singular curve $\tilde{\mathcal{R}}$, we have the following.

\begin{cor}\label{CorHantei}
Using the above notations and letting $g$ be the genus of $\mathcal{R}$,
the function $p\mapsto \psi((z,[\gamma]),w,p)$ has at most $g +  \displaystyle \sum_{\nu=1}^l \frac{m_\nu (m_\nu -1)}{2}$ poles.
\end{cor}

We note that $\varpi (\mathcal{R} ) = g +  \displaystyle \sum_{\nu=1}^l \frac{m_\nu (m_\nu -1)}{2}$ is called the arithmetic genus of  the algebraic curve $\mathcal{R}.$
If $\mathcal{R}$ is non-singular,  $g = \varpi (\mathcal{R})$ holds.

\subsection{A criterion for single-valued  differential operators}

From Proposition \ref{PAsz},
operators $Q_{(z,[\gamma])}$ of (\ref{Qz}),
which commutes with $P_z$ of (\ref{Pz}),
 can be multivalued on $\mathbb{C}-\mathcal{N}$.
However, they are sometimes single-valued on $\mathbb{C}-\mathcal{N}$.
In this subsection, we give a criterion for such  single-valued differential operators
by applying the  results of the eigenfunction $\psi$ in the previous subsection.

\begin{thm}\label{ThmCriteria}
For the differential operators $P_z$ and $Q_{(z,[\gamma])}$, 
assume the condition (\ref{assumez}).
Suppose $N$ is a prime number.
Let $\varpi (\mathcal{R})$ be the arithmetic genus of $\mathcal{R}$.
If $\varpi (\mathcal{R}) < N$,
every coefficient of $Q_{(z,[\gamma])}$ is single-valued on $\mathbb{C}-\mathcal{N}$.
\end{thm}

\begin{proof}
By the assumption (\ref{assumez}),
the eigenvalues $A(\lambda_1),\cdots,A(\lambda_N)$ are distinct.
We have the eigenfunction $\psi$ of (\ref{psiz}).
Due to Lemma \ref{LemPole},
we can take the base point $w$ such that 
there exist a pole $q$ of the function $p\mapsto \psi ((z,[\gamma]),w,p)$
which is not a ramification point of the projection $\pi:\mathcal{R}\rightarrow \mathbb{P}^1(\mathbb{C}).$

We assume that 
\begin{align}\label{assumesigma}
\text{there exists } [\delta_0]\in\pi_1 (\mathbb{C}-\mathcal{N}) \text{ such that } \sigma_{[\delta_0]}\not = id.
\end{align}
For generic $X \in \mathbb{P}^1(\mathbb{C})-\{\infty\}$
where $X$ is not a branch point of $\pi$ and
$\pi^{-1}(X)$ consists of $N$ distinct points $p_j = (X,A(\lambda_j))$ $(j=1,\cdots,N)$,
there are $k_0,k_1\in \{1,\cdots, N\}$ such that $k_0\not=k_1$ and 
\begin{align}\label{sigma0}
\sigma_{[\delta_0]}^{-1} (p_{k_0})=p_{k_1}.
\end{align}
From (\ref{sigmaD}) and (\ref{sigmaC}),
 (\ref{sigma0}) means that
 $
 \mathcal{Q}_{[\gamma],X} \Psi((z,[\gamma \cdot \delta_0]),w,\lambda_{k_0}) = A(\lambda_{k_1}) \Psi((z,[\gamma \cdot \delta_0]),w,\lambda_{k_0}) .
 $
 Since the eigenvalues of $ \mathcal{Q}_{[\gamma],X}$ are distinct,
 we obtain
  \begin{align}\label{kaiten}
 \Psi((z,[\gamma\cdot \delta_0]),w,\lambda_{k_0}) = {\rm const} \hspace{1mm} \Psi((z,[\gamma]),w,\lambda_{k_1}) 
 \end{align}
 for generic $(z,[\gamma])$ and $X$.
 Since $N$ is a prime number,
 by fixing the branch $\lambda$ of $\sqrt[N]{X}$ and letting $\zeta_N$ be the $N$-th root of the unity,
 we can suppose that
 $\lambda_{k_0} = \lambda$
 and $\lambda_{k_1}= \zeta_N^{l}  \lambda$
 for some $l\in \{0,\cdots, N-1\}$.  
 Recalling the form of $\Psi$ of (\ref{Psiz}),
the equation (\ref{kaiten})  induces the relation
 \begin{align}\label{kaiten1}
\Big(\displaystyle \sum_{s=0}^\infty  \xi_s((z,[\gamma\cdot \delta_0]),w) \lambda ^{-s} \Big) e^{\lambda (z-w)}
=
{\rm const } \hspace{1mm} \Big(\displaystyle \sum_{s=0}^\infty  \xi_s((z,[\gamma]),w) ((\zeta_{N}^l \lambda) ^{-s}) \Big) e^{(\zeta_N^l \lambda) (z-w)}
 \end{align}
for generic $(z,[\gamma])$ and $\lambda$.
By substituting
$ \zeta_N^l \lambda$  for $\lambda$,
we have
 $$
\Big(\displaystyle \sum_{s=0}^\infty  \xi_s((z,[\gamma\cdot \delta_0]),w)  (\zeta_N^l \lambda) ^{-s} \Big) e^{ (\zeta_N^l \lambda) (z-w)}
=
{\rm const } \hspace{1mm} \Big(\displaystyle \sum_{s=0}^\infty  \xi_s((z,[\gamma]),w) (\zeta_{N}^{2l} \lambda) ^{-s}\Big) e^{(\zeta_N^{2l} \lambda) (z-w)}
 $$
 from (\ref{kaiten1}).
 This means that
 it holds
 $
 \Psi((z,[\gamma\cdot \delta_0]),w,\lambda_{k_1})={\rm const } \hspace{1mm}\Psi((z,[\gamma]),w,\lambda_{k_2})
 $
  for generic $(z,[\gamma])$ and $\lambda$, where $\lambda_2 = \zeta_N^{2l} \lambda$.
  Setting $p_{k_2} =(X,A(\lambda_2))$, we have $\sigma_{[\delta_0^2]}^{-1} (p_{k_0})=p_{k_2}$ because
  $
   \Psi((z,[\gamma\cdot \delta_0^2]),w,\lambda_{k_0}) = {\rm const} \hspace{1mm} \Psi((z,[\gamma\cdot \delta_0]),w,\lambda_{k_1}) = {\rm const} \hspace{1mm} \Psi((z,[\gamma]),w,\lambda_{k_2})
  $
  holds.
 This implies that
 $ \sigma_{[\delta_0^2]}^{-1} (p_{k_0})=p_{k_2}$.
 Repeating this argument, 
 putting $p_m = (X,A(\lambda_{k_m}))$ 
  for $\lambda_{k_m}=\zeta_N^{ml} \lambda$,
 we have 
  \begin{align}\label{sigmadelta}
  \sigma_{[\delta_0^m]}^{-1} (p_{k_0})=p_{k_m} \quad \quad (m=0,\cdots,N-1).
  \end{align}
  Since $N$ is a prime number and $A(\lambda)$ is given by the form (\ref{star}),
  $p_{k_0},\cdots,p_{k_{N-1}}$ are distinct 
  and $\pi^{-1}(X)=\{p_{k_0},\cdots,p_{k_{N-1}}\}$. 
  Namely,
  (\ref{sigmadelta}) means that
  the action of the group $\langle \sigma_{[\delta_0]}\rangle$,
  which is  generated by $\sigma_{[\delta_0]}$,
  is transitive on the fibre $\pi^{-1}(X)$ for generic $X$.

  Recalling
 the eigenfunction $\psi$,
   (\ref{sigmadelta}) implies that
   \begin{align}\label{kaitenpsi}
   \psi((z,[\gamma \cdot \delta_0^m]),w,p)= {\rm const} \hspace{1mm}\psi((z,[\gamma]),w,\sigma_{[\delta_0^{m}]}^{-1} (p))
   \end{align}
    for $m=0,\cdots, N-1$,
   if $\pi(p)$ is not a branch point of $\pi$.
   At the beginning of the proof,
   we took the pole $q$ of the function
   $p \mapsto \psi((z,[\gamma]),w,p)$
   such that $\pi(q)$ 
   is not a branch point. 
   Since 
   the poles of $p\mapsto \psi((z,[\gamma]),w,p)$
   do not 
   depend on $(z,[\gamma])$
(see Lemma \ref{Lempsiz} (2)),
(\ref{kaitenpsi})
yields that
$\sigma_{[\delta_0^m]}(q)$ for $m\in \{0,\cdots,N-1\}$ are also poles.
So, we have at least $N$ distinct poles of $p\mapsto \psi((z,[\gamma]),w,p)$.

However, 
due to
the assumption,
 Theorem \ref{Thmphipolez} and its corollary,
we have at most
$\varpi (\mathcal{R}) ( < N)$ poles
of $p \mapsto \psi((z,[\gamma]),w,p)$.
This is a contradiction.
So, the assumption (\ref{assumesigma}) is false.
Therefore,
$\sigma_{[\delta]}=id$ for any $[\delta]\in \pi_1(\mathbb{C}-\mathcal{N})$.
According to Theorem \ref{ThmSingleSigma},
this means that all of the coefficients of $Q_{(z,[\gamma])}$ are single-valued. 
\end{proof}

\section{Differential equations with the action of the symplectic group}

\subsection{Preliminaries of automorphic forms}

For a commutative algebra $A$,
we set 
$Sp(n,A)=\{\alpha\in GL(2n,A)|{}^t \alpha J \alpha =J \}$,
where $J=\begin{pmatrix} 0 & -I_n \\ I_n & 0 \end{pmatrix}$.
The Siegel upper half plane $\mathbb{H}_n$ is given by
$
\mathbb{H}_n=\{\Omega\in M_n (\mathbb{C})  | {}^t\Omega = \Omega, {\rm Im}(\Omega)>0\}.
$
Here,
${\rm Im}(\Omega)>0$ means that the imaginary part of $\Omega$ is positive definite.
If $n=1$, $\mathbb{H}_1$ is the ordinary upper half plane $\mathbb{H}=\{z\in\mathbb{C}| {\rm Im}(z)>0\}$.
For $\alpha \in Sp(n,\mathbb{R})$  given 
by $\alpha = \begin{pmatrix} A & B \\ C & D\end{pmatrix}$, where $A,B,C,D\in M_n(\mathbb{R})$, 
we have the point
$\alpha (\Omega)= (A\Omega +B)(C\Omega + D)^{-1}\in \mathbb{H}_n$.
Set
$
j (\alpha,\Omega)={\rm det}(C\Omega+D).
$
We note that $j_\alpha (\Omega)\not =0.$

To define automorphic forms,
we will consider the case that the commutative ring $A$ is given by a totally real field $F$ such that $[F:\mathbb{Q}]=g.$
Let $\varphi_1,\cdots, \varphi_g : F \hookrightarrow \mathbb{R}$
be distinct $g$ embeddings.
Set ${\bf a}=\{\varphi_1,\cdots, \varphi_g \}.$
For any $\alpha\in Sp(n,F)$, let $\alpha^{\varphi_j} $ be the matrix 
whose components are given by the image of the components of $\alpha $ under $\varphi_j$. 
So, $\bf{a}$ embeds $Sp(n,F)$ to  $Sp(n,\mathbb{R})^g$
by
$\alpha \mapsto (\alpha^{\varphi_1},\cdots, \alpha^{\varphi_g}).$
From now on,
we will identify $Sp(n,F)$ with its image in  $Sp(n,\mathbb{R})^g$
via this embedding.
Then, for $\alpha=(\alpha^{\varphi_1},\cdots, \alpha^{\varphi_g})\in Sp(n,F)^g$
and
$\Omega=(\Omega_1,\cdots, \Omega_g)\in \mathbb{H}_n^g$,
we set
\begin{align}
\alpha(\Omega)=(\alpha^{\varphi_1}(\Omega_1),\cdots,\alpha^{\varphi_g}(\Omega_g))\in \mathbb{H}_n^g.
\end{align}
For any $\mathbb{C}$-valued function $f$ on $\mathbb{H}_n^g$ and $K\in \mathbb{Z}$,
we set
\begin{align}
f|_{[\alpha]_K}(\Omega)=j_{\alpha}(\Omega)^{-K} f(\alpha(\Omega)),
\end{align}
where 
$\displaystyle j_\alpha(\Omega)=\prod_{\nu=1}^g j(\alpha^{\varphi_j},\Omega_j).$
Throughout this paper,
we use these notations.

Let $\mathfrak{O}_F$ be the ring of integers of $F$.
For an ideal $\mathfrak{c}\subset \mathfrak{O}_F$,
we set 
$\Gamma(\mathfrak{c})=\{\alpha\in Sp(n,\mathfrak{O}_F)| \alpha-I_{2n}\in \mathfrak{c} M(2n,\mathfrak{O}_F)\}.$
For a group $\Gamma\subset Sp(n,F)$, if there exists an ideal $\mathfrak{c}$
such that $\Gamma$ contains $\Gamma(\mathfrak{c})$ as a finite index subset,
$\Gamma$ is called a congruence subgroup of $Sp(n,F).$

\begin{df}\label{dfAutomorph}
Let $\Gamma \subset Sp(n,F)$ be a congruence subgroup.
If a function $f$ on $\mathbb{H}_n^g$ satisfies the following conditions (i), (ii) and (iii),
we call $f$ an automorphic form for $\Gamma$ of weight $K$.

\begin{itemize} 

\item [(i)]
$f$ is holomorphic on $\mathbb{H}_n^g$.

\item[(ii)]
$f$ satisfies
$f|_{[\alpha]_K}=f$ for any $\alpha \in \Gamma$.

\item[(iii)]
When $F=\mathbb{Q}$ and $n=1$, 
$f|_{[\alpha]_K}(\Omega)$ has a  holomorphic Fourier expansion at cusps
for any $\alpha \in SL(2,\mathbb{Z})$.
Namely,
$
f|_{[\alpha]_K}(\Omega)=\displaystyle \sum_{k=0}^\infty \tilde{f}_{\alpha,k} \exp\Big(\frac{2 \pi \sqrt{-1} k \Omega}{N_\alpha}\Big),
$
holds,
where $\tilde{f}_{\alpha,k}\in \mathbb{C}$ and $N_\alpha \in \mathbb{Z}_{>0}.$
Here, `holomorphic' means that the Fourier expansion does not have any terms for $k<0$.
 \end{itemize}
\end{df}

\begin{rem}\label{RemFor}
The case of $F=\mathbb{Q} $ and $n=1$ is an exceptional case.
The condition (iii) is a growth condition for the cusps of $\Gamma$.
When $F\not = \mathbb{Q}$ or $n\geq 2$,
such a condition follows from the conditions (i) and (ii) (Koecher's principle).
\end{rem}

We note that automorphic forms of several variables are defined in various literature.
Our definition above of automorphic forms is 
due to 
\cite{S95}.
This definition seems  general enough for applications 
because
 we can obtain  important  modular functions as reductions.
For example, if  $\Gamma=Sp(n,\mathbb{Q})$,
then the corresponding automorphic forms are  well-known Siegel modular forms.
If $F\not=\mathbb{Q}$ and $n=1$,  
then the corresponding automorphic forms are Hilbert modular forms.

\subsection{Differential operators with coefficients satisfying a transformation law}

Let $a_l(\Omega,z)$ $(l=2,\cdots,N)$ be a function
of $\Omega=(\Omega_1,\cdots,\Omega_g)\in\mathbb{H}_n^g$ and $z\in \mathbb{C}$.
We suppose that $\Omega \mapsto a_l(\Omega,z)$  is holomorphic for generic $z$.
Moreover,
for fixed $\Omega$,
let $z\mapsto a_l(\Omega,z)$ be an  analytic function  
with at most poles.
We consider the cases that $a_l(\Omega,z)$ satisfies the transformation law  
\begin{align}\label{Trans}
a_l\left(\alpha(\Omega),\frac{z}{j_\alpha (\Omega)}\right)= j_\alpha(\Omega)^l a_l(\Omega,z),
\end{align}
for $\alpha \in \Gamma$.
For fixed $\Omega\in \mathbb{H}_n^g$,
let $\mathcal{N}_\Omega \subset \mathbb{C}=(z \text{-plane})$ be the union of  the sets of poles of the function $a_j(\Omega,z)$ $(j=2,\cdots,n)$.
Namely, for a fixed $\Omega\in \mathbb{H}_n^g$,
$a_2(\Omega,z),\cdots,a_N(\Omega,z)$ are holomorphic functions of $z \in \mathbb{C}-\mathcal{N}_\Omega$.

\begin{rem}
If $n=1$,
the action $(\Omega,z)\mapsto \Big(\alpha(\Omega),\frac{z}{j_\alpha (\Omega)}\Big)$ 
is equal to the action which defines the Jacobi forms of degree $1$ (see \cite{EZ}).
However,
if $n\geq 2$,
our action is different from the action for Jacobi forms of higher degrees studied in \cite{Ziegler}.
\end{rem}

\begin{lem}\label{Lempoles}
For any $\alpha\in \Gamma$ and $\Omega\in \mathbb{H}_n^g$,
 $z\in \mathbb{C} - \mathcal{N}_\Omega$ if and only if $\frac{z}{j_\alpha (\Omega)} \in \mathbb{C}- \mathcal{N}_{\alpha(\Omega)}. $
\end{lem}

\begin{proof}
Due to the transformation law (\ref{Trans}),
it holds that 
\begin{align*}
z\in \mathbb{C}-\mathcal{N}_\Omega &\Longleftrightarrow a_l(\Omega,z)\not = \infty \quad (l=2,\cdots, N)\\
&\Longleftrightarrow a_l\Big(\alpha(\Omega),\frac{z}{j_\alpha(\Omega)}\Big)=j_\alpha (\Omega)^l a_l(\Omega,z)\not = \infty. \\
&\Longleftrightarrow \frac{z}{j_\alpha(\Omega)}\in \mathbb{C}-\mathcal{N}_{\alpha (\Omega)}.
\end{align*}
\end{proof}

Let $\mathfrak{X}_\Omega$ be the universal covering of $\mathbb{C}-\mathcal{N}_\Omega.$
For a fixed point $w\in \mathbb{C}-\mathcal{N}_\Omega$,
any $s\in \mathfrak{X}_\Omega$ is represented by
$
s=(z,[\gamma]),
$
where $z\in \mathbb{C}-\mathcal{N}_\Omega$, $\gamma$ is an arc in $\mathbb{C}-\mathcal{N}_\Omega$ from $w$ to $z$ 
and $[\gamma]$ is the homotopy class of $\gamma$.
We note that $z $ gives a local coordinate of $\mathfrak{X}_\Omega$.

Let us consider the following  ordinary differential operator of the independent variable $z$: 
\begin{align}\label{MotherEq}
P_{\Omega,z}=
\frac{\partial^{N}}{\partial z^{N}} +a_2(\Omega,z) \frac{\partial^{N-2}}{\partial z^{N-2}} +a_3(\Omega,z)\frac{\partial ^{N-3}}{\partial z^{N-3}} +\cdots + a_N(\Omega,z).
\end{align}
Set
$(\Omega_1,z_1)=(\alpha (\Omega), \frac{z}{j_\alpha (\Omega)}) $.
Throughout this paper,
we assume that $\Omega\mapsto a_l(\Omega,z)$ are holomorphic for generic $z$.
Since 
$\frac{\partial }{\partial z_1}=j_\alpha (\Omega) \frac{\partial }{\partial z}$ and $a_l(\Omega_1,z_1)= j_\alpha (\Omega)^l a_l(\Omega,z) $,
we have
\begin{align}\label{WeightDiff}
P_{\Omega_1,z_1}=j_\alpha (\Omega)^{N} P_{\Omega,z}.
\end{align}

\begin{df}\label{DfDiff}
Let $D_{\Omega,z}  $ be a linear differential operator  of $z$ holomorphically parametrized  by $\Omega\in \mathbb{H}_n^g$.
If 
\begin{align}
D_{\Omega_1,z_1}=j_\alpha (\Omega)^K D_{\Omega,z}
\end{align}
holds for $\alpha \in \Gamma$,
we call $D_{\Omega,z}$ a differential operator of weight $K$ with respect to the action of $\Gamma$.
We call $D_{\Omega,z}u=0$ a linear differential equation of weight $K$ with respect to the action of $\Gamma$.
\end{df}

There exist many important examples which satisfy the transformation law  (\ref{Trans}).

\begin{exap}
Let $\Gamma$ be a congruence subgroup of $SL(n,F)$.
If $f(\Omega)$ be an automorphic form of weight $j$,
then $a_{j+k}(\Omega,z)=z^{-k} f(\Omega)$ satisfies the transformation law (\ref{Trans}) for $l=j+k$ for any $k\in \mathbb{Z}_{\geq 0}$.
If $k>0$, then
$\mathcal{N}_\Omega=\{0\}$ holds.
\end{exap}

\begin{exap}\label{ExapJacobi}
For a congruence subgroup $\Gamma \subset SL(2,\mathbb{Z}),$
a weak Jacobi form $\mathbb{H} \times \mathbb{C} \ni (\Omega,z) \mapsto f(\Omega,z)\in \mathbb{C}$ for $\Gamma \subset SL(2,\mathbb{Z})$ of weight $K$ and level $m$ is a holomorphic  function with the following properties
\begin{itemize}

\item [(i)] 
for any $\alpha=\begin{pmatrix} a & b \\ c & d \end{pmatrix} \in \Gamma, $
$f(\alpha (\Omega),\frac{z}{j_\alpha (\Omega)}) = j_\alpha(\Omega)^K {\rm exp}(-2\pi i \frac{mc z^2}{j_\alpha (\Omega)}) f(\Omega,z)$,

\item[(ii)]
for any $n_1,n_2 \in \mathbb{Z}$
$f(\Omega,z+n_1 \Omega + n_2)={\rm exp}(-\pi i (n_1^2 \Omega +2n_1 z)) f(\Omega,z)$,

\item[(iii)]
$f$ has a Fourier expansion $f(\Omega,z)=\displaystyle \sum_{n,l\in \mathbb{Z} } c_{n,l} {\rm exp}\Big(\frac{2\pi \sqrt{-1} n\Omega}{N}\Big) {\rm exp}(2\pi \sqrt{-1} n z )$ for some $N\in \mathbb{Z}$.
 \end{itemize}
Weak Jacobi forms are very important in number theory (see \cite{EZ}). 
If $f(\Omega,z)$ ($g(\Omega,z)$, resp.) is a weak Jacobi form for $\Gamma$ of weight $K_1$ ($K_2$, resp.) and level $m$,
then $a_l(\Omega,z)=\frac{f(\Omega,z)}{g(\Omega,z)}$ satisfies the transformation law of (\ref{Trans}) for $n=g=1$ and $l=K_1 - K_2$.
 \end{exap}
 
 \begin{exap}\label{ExapLame}
As a special case of \ref{ExapJacobi},
we consider  the  Lam\'e differential operator
 \begin{align}\label{Lame}
 P_{\Omega,z}=\frac{\partial^2}{\partial z^2} -B \wp(\Omega,z),
 \end{align}
 where $\Omega\in \mathbb{H}$ and $\wp(\Omega,z)$ is the Weierstrass $\wp$-function
 $$
 \wp(\Omega,z)=\frac{1}{z^2} + \displaystyle \sum_{(n_1,n_2)\in\mathbb{Z}^2 -\{(0,0)\}}\Big( \frac{1}{(z-n_1 -n_2 \Omega)^2} - \frac{1}{(n_1 + n_2 \Omega)^2}\Big).
 $$
  We note that $z\mapsto \wp (\Omega,z)$ has poles of degree $2$ at every $z_0\in \mathcal{N}_\Omega:=\mathbb{Z}+\mathbb{Z} \Omega$.
  
  Let $\Gamma$ be the elliptic full-modular group $SL(2,\mathbb{Z})$.
  For any $\alpha\in\Gamma$,
  we have 
  $$
  \wp\Big(\alpha(\Omega),\frac{z}{j_\alpha (\Omega)}\Big)=j_\alpha (\Omega)^2 \wp(\Omega,z).
  $$
  Especially, 
  $\wp(\Omega +1,z)=\wp(\Omega,z)$ holds and $\wp $ has the Fourier expansion at cusps:
  $$
  \wp(\Omega,z)=\pi^2 \Big(\frac{1}{6} - 4 \displaystyle \sum_{n=1}^\infty \frac{n q^{2n}}{1-q^{2n}} \Big)+ \frac{\pi^2}{\sin^2 (\pi z)}-8\pi^2 \sum_{n=1}^\infty \cos (2n\pi z) \frac{n q^{2n}}{1-q^{2n}},
  $$
   where $q={\rm exp}(2\pi \sqrt{-1} \Omega) $ (for detail, see \cite{EMOF}). 
  Therefore, 
  in terms  of Definition \ref{DfDiff},
 $P_{\Omega,z}$  of (\ref{Lame})
    is a differential operator of weight $2$ with respect to the action of $\Gamma=SL(2,\mathbb{Z}).$
 \end{exap}

When $\gamma:[0,1]\rightarrow \mathbb{C}-\mathcal{N}_\Omega$ is an arc with $\gamma(0)=w$ and $\gamma(1)=z$,
 let  $\gamma_1=\frac{\gamma}{j_\alpha (\Omega)}$  be the arc given by $\gamma_1(t)=\frac{\gamma(t)}{j_\alpha (\Omega)}$.
By virtue of Lemma \ref{Lempoles},
$\gamma_1$ is  an arc in $\mathbb{C}-\mathcal{N}_{\alpha(\Omega)}$.

\begin{thm}\label{ThmPsi}
Let $P_{\Omega,z}$ be the differential operator of (\ref{MotherEq}).

(1)
There exists the unique formal solution 
$\Psi(\Omega,(z,[\gamma]),w,\lambda)$
of the differential equation
\begin{align}\label{sp-diff}
P_{\Omega,z} u = \lambda^N u
\end{align}
in the form
\begin{align}\label{Psi}
\Psi(\Omega,(z,[\gamma]),w,\lambda )=\Big(\sum_{s=0}^\infty  \xi_s(\Omega,(z,[\gamma]),w) \lambda^{-s}\Big) e^{\lambda(z-w)}
\end{align}
such that
\begin{align}\label{ConditionPsi}
\begin{cases}
& \xi_0(\Omega,(z,[\gamma]),w)\equiv 1,\\
& \xi_s(\Omega,(w,[id]),w) = 0 \quad (s\geq 1).
\end{cases}
\end{align}
Here, $\Omega \mapsto \xi_s (\Omega,(z,[\gamma]),w)$ are  holomorphic for generic $((z,[\gamma]),w)$. 
Moreover, for a fixed $\Omega\in \mathbb{H}_n^g$,
$((z,[\gamma]),w) \mapsto \xi_s (\Omega,(z,[\gamma]))$ are locally holomorphic. 

(2)
For any $\alpha\in \Gamma$, it holds
\begin{align}\label{Psino0}
&\Psi\Big(\alpha(\Omega),\Big(\frac{z}{ j_\alpha (\Omega ) },\Big[\frac{\gamma}{ j_\alpha (\Omega ) }\Big]\Big), \frac{w}{ j_\alpha (\Omega ) } , j_\alpha (\Omega ) \lambda\Big)=\Psi(\Omega,(z,[\gamma]),w,\lambda). 
\end{align}
The  function $ \xi_s(\Omega,(z,[\gamma]),w)$ in
(\ref{Psi}) 
satisfies
the transformation law
\begin{align}\label{XiTrans}
\xi_s \left(\alpha (\Omega), \left(\frac{z}{j_\alpha(\Omega)},\left[\frac{\gamma}{j_\alpha(\Omega)}\right]\right), \frac{w}{j_\alpha (\Omega)} \right) = j_\alpha (\Omega)^s \xi_s(\Omega,(z,[\gamma]),w).
\end{align}
 \end{thm}

\begin{proof}
(1)
For fixed $\Omega \in \mathbb{H}_g,$
  putting $\displaystyle u=\Big(\sum_{s=0}^\infty \eta_s(\Omega,z) \lambda^{-s} \Big) e^{\lambda(z-w)}$ to (\ref{sp-diff}),
  by the same argument as in the proof of Proposition \ref{PropPsiz},
we can obtain
\begin{align}\label{Induction}
&N\frac{\partial }{\partial z} \eta_{N+s_0-1} (\Omega,z) \notag\\
&= \left(\text{a polynomial in } \frac{\partial^\nu}{\partial z^\nu}\eta_l (\Omega,z) \hspace{1mm} (l<N+s_0-1,\nu \in \mathbb{Z}_{\geq 0})\text{ and } a_j (\Omega,z) \text{ defined over } \mathbb{Z}\right)
\end{align}
for any $s_0$.
By the integration of the relation (\ref{Induction}) on arc $\gamma \in \mathbb{C}-\mathcal{N}_\Omega$ whose start point is $w$,
we can  obtain the expression of $\eta_\mu (\Omega,z)$ in terms of $\eta_\nu (\Omega,z)$ $(\nu < \mu)$ and $a_l(\Omega,z)$.  
From
the conditions that $\eta_0\equiv 1$
and
$\eta_s(\Omega,(w,[id]))=0$ $(s\geq1)$,
we can 
determines the sequence $\{\eta_s(\Omega,z)\}_s$ uniquely.
Such  $\eta_s(\Omega,(z,[\gamma]))$ give the required functions $\xi_s(\Omega,(z,[\gamma]),w)$.

Moreover,
since the coefficients $a_l(\Omega,z)$ of  (\ref{MotherEq}) are holomorphic functions of $\Omega$ for generic $z$
and $\xi_s (\Omega,(z,[\gamma]),w)$ are determined by 
  the construction   via (\ref{Induction}), 
$\xi_s (\Omega,(z,[\gamma]),w)$ are holomorphic functions of $\Omega\in \mathbb{H}_n^g$ for generic $((z,[\gamma]),w)$.
Also, 
for fixed $\Omega$,
$\xi_s (\Omega,(z,[\gamma]),w)$ are locally holomorphic functions of $((z,[\gamma]),w)\in \mathfrak{X}_\Omega \times (\mathbb{C}-\mathcal{N}_\Omega)$.

(2)
We consider the transformation
\begin{align}\label{11111}
(\Omega,(z,[\gamma]),w,\lambda)\mapsto (\Omega_1,(z_1,[\gamma_1]),w_1,\lambda_1)=\left(\alpha(\Omega),\left(\frac{z}{j_\alpha (\Omega)},\left[\frac{\gamma}{j_\alpha (\Omega)}\right]\right),\frac{w}{j_\alpha (\Omega)},j_\alpha (\Omega)\lambda \right).
\end{align}
By virtue of (\ref{WeightDiff}), the differential equation $P_{\Omega,z}u=\lambda^N u$ 
gives the same equation with $P_{\Omega_1,z_1}u=\lambda_1^N u$
under  the correspondence (\ref{11111}).
Since we have the uniqueness of the solution $\Psi $ in the form of (\ref{Psi}) and the condition (\ref{ConditionPsi}),
we obtain (\ref{Psino0}).
Then, we have
\begin{align} \label{Psino1}
 \Big(\sum_{s=0}^\infty  \xi_s(\Omega,(z,[\gamma]),w) \lambda^{-s}\Big) e^{\lambda(z-w)}=\Big(\sum_{s=0}^\infty  \xi_s(\Omega_1,(z_1,[\gamma_1]),w_1) \lambda_1^{-s}\Big) e^{\lambda_1 (z_1-w_1)}.
\end{align}
By cancelling $e^{\lambda (z-w)}=e^{\lambda_1 (z_1-w_1)}$ and comparing the coefficient of $\lambda^{-s}=\lambda_1^{-s} j_\alpha (\Omega)^s$,
we have the transformation law (\ref{XiTrans}).
\end{proof}

\subsection{Commutative differential operators with an action of a symplectic group (generic cases of $F\not= \mathbb{Q}$ or $n\geq 2$)}

We consider the differential operator
\begin{align}\label{Q}
Q_{\Omega,(z,[\gamma])}
=b_0(\Omega,(z,[\gamma])) \frac{\partial^M}{\partial z^M} +b_{1}(\Omega,(z,[\gamma])) \frac{\partial^{M-1}}{\partial z^{M-1}} + \cdots + b_M(\Omega,(z,[\gamma])),
\end{align}
which commutes with the differential operator $P_{\Omega,z}$ of (\ref{MotherEq}).
Here, we assume that 
 the coefficients $b_k(\Omega,(z,[\gamma]))$ $(k=0,\cdots, M)$ are locally analytic functions  of
  $ (z,[\gamma]) \in \mathfrak{X}_\Omega$.

\begin{thm}\label{PAs}
(1) Let $P_{\Omega,z}$ ($Q_{\Omega,(z,[\gamma])}$, resp.) be the differential operator of (\ref{MotherEq}) ((\ref{Q}), resp.).
Then, $P_{\Omega,z}$ and $Q_{\Omega,(z,[\gamma])}$ are commutative if and only if 
the quotient
$\displaystyle
\frac{Q _{\Omega,(z,[\gamma])}\Psi(\Omega,(z,[\gamma]),w,\lambda)}{\Psi(\Omega,(z,[\gamma]),w,\lambda)} $
for $\Psi$ of (\ref{Psi})
coincides with 
\begin{align}\label{starstar}
A(\Omega,\lambda)=\displaystyle \sum_{s=-M}^\infty A_s(\Omega)\lambda^{-s}
\end{align}
for generic $\lambda$,
where $A(\Omega,\lambda)$ does not depend on the variables $z$ and $w$.

(2) If $P_{\Omega,z}$ commutes with both $Q^{(1)}_{\Omega,(z,[\gamma])}$ and $Q^{(2)}_{\Omega,(z,[\gamma])}$,
then $Q^{(1)}_{\Omega,(z,[\gamma])}$ commutes with $Q^{(2)}_{\Omega,(z,[\gamma])}$.

(3)
If the differential operator $Q_{\Omega,(z,[\gamma])}$ is of weight $K$
with respect to the action of $\Gamma$,
then the members of the sequence $\{A_s(\Omega)\}_s$
satisfy
\begin{align}\label{AAMod}
A_s(\alpha (\Omega))=j_\alpha (\Omega)^{K+s} A_s(\Omega).
\end{align}
\end{thm}

\begin{proof}
(1) (2) 
These are proved by a similar argument to the proof of Propositon \ref{PAsz} and Proposition \ref{PQQCom}.

(3)
We recall that $\Psi$ in (\ref{Psi})
satisfies (\ref{Psino0}).
Since $Q_{\Omega,(z,[\gamma])}$ is of weight $K$,
we have
$
A(\Omega_1,\lambda_1) = j_\alpha (\Omega)^K A(\Omega,\lambda).
$
Namely, we have
\begin{align*}
\displaystyle \sum_{s=-M}^\infty A_s(\Omega_1) \lambda_1^{-s} = j_\alpha (\Omega)^K \displaystyle \sum_{s=-M}^\infty A_s(\Omega) \lambda^{-s}.
\end{align*}
By comparing the coefficients of $\lambda^{-s}$, the assertion follows.
\end{proof}

\begin{thm}\label{AutoToDiffZ}
For any $j\in \{0,-1,\cdots , -M\},$
let $A_{j}(\Omega)$ satisfy the transformation law
\begin{align}\label{ModTrans}
A_j(\Omega_1) = j_\alpha (\Omega)^{K+j} A_j(\Omega)
\end{align}
for any $\alpha\in \Gamma$.
If there exists a differential operator 
$Q_{\Omega,(z,[\gamma])}$ of rank $M$ of weight $K$
with respect to $\Gamma$
satisfying
  \begin{align}\label{QPAP2}
 Q_{\Omega,(z,[\gamma])} \Psi(\Omega,(z,[\gamma]),w,\lambda) = A(\Omega,\lambda)  \Psi(\Omega,(z,[\gamma]),w,\lambda),
 \end{align}
 where $\displaystyle A(\Omega,\lambda)=\sum_{s=-M}^\infty A_s(\Omega) \lambda^{-s}$,
 then $Q_{\Omega,(z,[\gamma])}$ 
 is uniquely determined only by 
given operator $P_{\Omega,z}$ of (\ref{MotherEq}) and
the functions  $A_{j}(\Omega)$ $(j=0,\cdots , -M)$. 
 Here,
 $A_s(\Omega)$ $(s\geq 1)$ are also uniquely determined only by  $P_{\Omega,z}$  and
  $A_{j}(\Omega)$ $(j=0,\cdots , -M)$.
 \end{thm}

\begin{proof}
Let $\Psi(\Omega,(z,[\gamma]),w,\lambda)$  be the solution of (\ref{Psi}) for the differential equation $P_{\Omega,z} u = X u$, where $X=\lambda^N$.
Let $\{A_s(\Omega)\}_s$ be the sequence satisfying 
the relation (\ref{AAMod}) 
and 
set
$A(\Omega,\lambda)=\displaystyle \sum_{s=-M}^\infty A_s(\Omega) \lambda^{-s}.$
If there exists
a differential operator $Q_{\Omega,(z,[\gamma]),w}$ satisfying (\ref{QPAP2}),
 then $Q_{\Omega,(z,[\gamma]),w}$ is a differential operator
 of weight $K$ with respect to $\Gamma$
 and commutes with $P_{\Omega,z}$.
 Next,
 taking $w'\in \mathbb{C}-\mathcal{N}_\Omega$ and another solution $\Psi(\Omega,(z,[\gamma]),w',\lambda),$
 we suppose that there is an operator $Q_{\Omega,(z,[\gamma],w')}$ such that
 \begin{align}\label{QPAP22}
 Q_{\Omega,(z,[\gamma]),w'} \Psi(\Omega,(z,[\gamma]),w',\lambda) = A(\Omega,\lambda)  \Psi(\Omega,(z,[\gamma]),w',\lambda).
 \end{align}
 As in Proposition \ref{PropUPAz} and Proposition \ref{PAsz},
 there exists $B(\Omega,\lambda)$ such that
 $\Psi(\Omega,(z,[\gamma]),w',\lambda) e^{\lambda (w'-w)}= B(\Omega,\lambda) \Psi (\Omega,(z,[\gamma]),w,\lambda).$
 Therefore, by (\ref{QPAP2}) and (\ref{QPAP22}),
 \begin{align*}
 &\frac{Q_{\Omega,(z,[\gamma]),w'} \Psi(\Omega,(z,[\gamma]),w,\lambda) }{\Psi(\Omega,(z,[\gamma]),w,\lambda) } = \frac{Q_{\Omega,(z,[\gamma]),w'} B(\Omega,\lambda) \Psi(\Omega,(z,[\gamma]),w,\lambda) e^{\lambda (w-w')} }{B(\Omega,\lambda)\Psi(\Omega,(z,[\gamma]),w,\lambda)e^{\lambda (w-w')} }\\
 &=\frac{Q_{\Omega,(z,[\gamma]),w'} \Psi(\Omega,(z,[\gamma]),w',\lambda) }{\Psi(\Omega,(z,[\gamma]),w',\lambda) }
 =A(\Omega,\lambda) =\frac{Q_{\Omega,(z,[\gamma]),w} \Psi(\Omega,(z,[\gamma]),w,\lambda) }{\Psi(\Omega,(z,[\gamma]),w,\lambda)}.
 \end{align*}
 Therefore,
 we obtain
 $(Q_{\Omega,(z,[\gamma]),w}-Q_{\Omega,(z,[\gamma]),w'})\Psi(\Omega,(z,[\gamma]),w,\lambda)=0.$
 By a similar argument to  the end of the proof of Proposition \ref{PAsz},
 we can see that $Q_{\Omega,(z,[\gamma]),w} = Q_{\Omega,(z,[\gamma]),w'}$.
 Hence,
  a differential operator $Q_{\Omega,(z,[\gamma]),w}$ satisfying  (\ref{QPAP2}) does not depend on the base point $w.$
 So,
 we use the notation $Q_{\Omega,(z,[\gamma])}$ instead of $Q_{\Omega,(z,[\gamma]),w}$.

 Now, we  see that  the differential operator $Q_{\Omega,(z,[\gamma])}$
 and the series $A(\Omega,\lambda)$
  satisfying (\ref{QPAP2})
  are uniquely determined by $P_{\Omega,z}$ and $A_0(\Omega),\cdots, A_{-M}(\Omega)$.
The relation (\ref{QPAP2})
is equal to
\begin{align} \label{bxAx}
&\displaystyle \sum_{s=0}^\infty \displaystyle \sum_{k=0}^M\displaystyle \sum_{\alpha=0}^k \begin{pmatrix}  k \\ \alpha \end{pmatrix} b_{M-k}(\Omega,(z,[\gamma])) \frac{\partial^\alpha}{\partial z^\alpha}\xi_s(\Omega,(z,[\gamma]),w) \lambda^{k-\alpha-s} \notag \\
&= \left(\displaystyle \sum_{t=-M}^\infty A_t (\Omega) \lambda^{-t}\right)\left(\displaystyle \sum_{s=0}^\infty \xi_s(\Omega,(z,[\gamma]),w)\lambda^{-s}\right).
\end{align}
We note that $\{\xi_s\}_s$ is determined only by the given differential operator $P_{\Omega,z}$ by Theorem \ref{ThmPsi} (1).
For any $j\in \{0,1,\cdots,M\}$, 
recalling that $\xi_0\equiv 1 $ and taking the coefficients of the term for $\lambda^{M-j}$ in the equation (\ref{bxAx}),
we have
\begin{align}\label{bbxAx}
&\notag b_{j}(\Omega,(z,[\gamma])) + \Big(\text{a polynomial in } b_k(\Omega,(z,[\gamma])) \hspace{1mm}(k\leq j-1) \text{ and } \frac{\partial^\nu}{\partial z^\nu}\xi_s(\Omega,(z,[\gamma]),w) \hspace{1mm} (\nu\in\mathbb{Z}_{\geq 0})\Big)\\
&= \Big(\text{a polynomial in } A_{t}(\Omega) \hspace{1mm} (t\leq j-M) \text{ and } \xi_s(\Omega,(z,[\gamma]),w)  \Big).
\end{align}
From (\ref{bbxAx}),
we can obtain $b_{j}(\Omega,(z,[\gamma]))$ $(j=0,1,\cdots, M)$ inductively.
This shows that $A_{-M}(\Omega),\cdots, A_0(M)$ and the differential operator $P_{\Omega,z}$ uniquely determine  $Q_{\Omega,(z,[\gamma])}$.
Moreover,
since $\xi_0\equiv 1$ again,
from the coefficients of $\lambda^{-s}$ in the relation (\ref{bxAx}),
we have
\begin{align}\label{AxAx}
& A_s (\Omega) +  \Big(\text{a polynomial in } A_{t}(\Omega) \hspace{1mm} (t < s) \text{ and } \frac{\partial ^\nu}{\partial z^\nu}\xi_s(\Omega,(z,[\gamma]),w) \hspace{1mm} (\nu\in\mathbb{Z}_{\geq 0})\Big) \notag \\
&= \Big(\text{a polynomial in } b_k(\Omega,(z,[\gamma])) \hspace{1mm}(0\leq k\leq M) \text{ and } \xi_s(\Omega,(z,[\gamma]),w) \Big),
\end{align}
for any $s\geq -M$.
Especially, $A_s(\Omega)$ $(s\geq 1)$ are inductively determined by  (\ref{AxAx}).
Here,
we note that such $A_s(\Omega)$ $(s\geq 1)$ do not depend on $z$ and $w$ by virtue of Theorem \ref{PAs} (1).
Therefore, the assertion follows.
\end{proof}

The following theorem 
gives a correspondence 
between 
automorphic forms 
and
differential operators which commutes with $P_{\Omega,z}$
for generic cases of $F\not = \mathbb{Q}$ or $n\geq 2$.

\begin{thm} \label{AutoToDiff}
Suppose $F\not = \mathbb{Q} $ or $n\geq 2$.
Let $P_{\Omega,z}$  of (\ref{MotherEq}) ($Q_{\Omega,(z,[\gamma])}$ of (\ref{Q}), resp.) 
be differential operators
studied  in Theorem  \ref{AutoToDiffZ}.

(1) If $A_j(\Omega)$ $(j=0,\cdots, -M)$ 
are automorphic forms of weight $K+j$ for $\Gamma$,
then any coefficients $b_j(\Omega,(z,[\gamma]))$ of $Q_{\Omega,(z,[\gamma])}$
($A_s(\Omega)$ ($s\geq 1$) of $A(\Omega,\lambda)$, resp.),
which is derived from $P_{\Omega,z}$  and $A_j(\Omega)$ $(j=0,-1,\cdots, -M)$,
give holomorphic functions  $\Omega \mapsto b_j(\Omega,(z,[\gamma]))$ for generic $(z,[\gamma])$
(automorphic forms  of weight $K+s$ for $\Gamma$, resp.).

(2) Conversely,
if every 
coefficient  $b_j(\Omega,(z,[\gamma]))$ of $Q_{\Omega,(z,[\gamma])}$
gives a holomorphic function   $\Omega\mapsto b_j(\Omega,(z,[\gamma]))$ for generic $(z,[\gamma]),$
then $A_s(\Omega) $ $(s\geq -M)$,
which are determined in the sense of Theorem \ref{PAs}, 
are automorphic forms for $\Gamma$.
\end{thm}

\begin{proof}
(1) Recall Definition \ref{dfAutomorph}.
From the assumption, $A_{j}(\Omega)$ $(k\in \{0,\cdots,-M\})$ are holomorphic function of $\Omega\in \mathbb{H}_n^g$
satisfying the transformation law (\ref{ModTrans}). 
From Theorem \ref{ThmPsi}, $\xi_s(\Omega,(z,[\gamma]),w)$ are holomorphic of $\Omega\in \mathbb{H}_n^g$ for generic $((z,[\gamma]),w)$.
So, due to the construction of $b_j(\Omega,(z,[\gamma]))$ via the relation (\ref{bbxAx}),
$b_j(\Omega,(z,[\gamma]))$ are holomorphic of $\Omega$ for generic $(z,[\gamma])$.  
Also, from (\ref{AxAx}),
$A_s(\Omega) $ $(s\geq 1)$ are also holomorphic in $\Omega$.
Moreover, by Theorem \ref{PAs},
we obtain
$A_s(\Omega)$ $(s\geq 1)$ satisfying the transformation law (\ref{AAMod}).
So,  from Definition   \ref{dfAutomorph}, $A_s(\Omega)$ $(s\geq 1)$  are automorphic forms for $\Gamma$ of weight $K+s$.

(2)
From Theorem \ref{ThmPsi} and Theorem \ref{PAs},
we only need to see that $\Omega \mapsto A_s(\Omega)$
are holomorphic under our assumption.
However,
we can  see this property,
because $A_0(\Omega),\cdots,A_{-M}(\Omega)$ are determined
by $b_j(\Omega,(z,[\gamma]))$ via (\ref{AxAx})
and do not depend on $(z,[\gamma])$ and $w$. 
\end{proof}

\subsection{Commutative differential operators with an action of a symplectic group (exceptional cases of $F= \mathbb{Q}$ and $n =1$)}

In this subsection,
we consider exceptional cases of $F=\mathbb{Q}$ and $n=1$ carefully.
In such cases,
we need to consider the action of $\alpha \in SL(2,\mathbb{Z})$,
because automorphic forms for such exceptional cases
require holomorphic Fourier expansion at cusps in the sense of 
 Definition \ref{dfAutomorph} (iii).

Recall that
the set of poles of $P_{\Omega,z}$ of (\ref{MotherEq}) is 
given by $\mathcal{N}_{\Omega}$.
If the coefficients $a_l(\Omega,z)$ $(l=2,\cdots,N)$ of $P_{\Omega,z}$
satisfy the transformation law (\ref{Trans}) for $\alpha\in SL(2,\mathbb{Z})$,
we have $\mathcal{N}_\Omega = \mathcal{N}_{\alpha (\Omega)}$  for any $\alpha \in SL(2,\mathbb{Z})$
by virtue of Lemma \ref{Lempoles}.
However,
the transformation law (\ref{Trans}) for $\alpha \in SL(2,\mathbb{Z})$  generically does not hold.
So, we need a bit delicate argument for holomorphic Fourier expansions at cusps.
Let $j_\alpha (\Omega) \cdot (\mathbb{C}-\mathcal{N}_{\alpha (\Omega)})$ be the set
$\{j_\alpha (\Omega) z_1| z_1\in \mathbb{C}-\mathcal{N}_{\alpha (\Omega)} \}$.
If
an arc $\gamma$ is in  $j_\alpha (\Omega) \cdot (\mathbb{C}-\mathcal{N}_{\alpha (\Omega)})$,
then
$\gamma_1 = \frac{\gamma}{j_\alpha (\Omega)}$
is in $\mathbb{C}-\mathcal{N}_{\alpha(\Omega)}$.
We have the following lemma.

\begin{lem}\label{LemFourier}
Suppose $F=\mathbb{Q}$ and $n=1$.
For any $\alpha\in SL(2,\mathbb{Z})$,
we suppose that 
the coefficients $a_l(\Omega,z)$ $(l=2,\cdots,N)$
have the holomorphic Fourier expansion at cusps:
 \begin{align}\label{Fouriera}
  \displaystyle 
 j_\alpha (\Omega)^{-l} a_l\Big(\alpha(\Omega), \frac{z}{j_\alpha (\Omega)} \Big)=\sum_{k\geq 0} \tilde{a}_{l,\alpha,k} (z) \exp\left(\frac{2\pi \sqrt{-1} k \Omega}{N_{l,\alpha}}\right).
 \end{align}
 where $\tilde{a}_{l,\alpha,k} (z)$ are holomorphic functions of $z\in \mathbb{C}-\mathcal{N}_{\alpha(\Omega)}$ and $N_{l,k}\in \mathbb{Z}_{>0}.$
 Here `holomorphic' means that the expression (\ref{Fouriera}) does not contain any terms for $k<0$.
Then,
 the coefficients $\xi_s$ of the multivalued Baker-Akhiezer function $\Psi$ of (\ref{Psi})
has a holomorphic Fourier expansion  at cusps:
\begin{align}\label{Fourierxi}
j_\alpha(\Omega)^{-s}\xi_s\left(\alpha(\Omega),\left(\frac{z}{j_\alpha(\Omega)},\left[\frac{\gamma}{j_\alpha (\Omega)} \right]\right),\frac{w}{j_\alpha (\Omega)}\right)=\displaystyle \sum_{k\geq 0} \tilde{\xi}_{s,\alpha,k}\left(\left(z,[\gamma]\right),w\right)\exp \left(\frac{2\pi \sqrt{-1} k\Omega}{N_{s,\alpha}}\right),
\end{align}
where
$\tilde{\xi}_{s,\alpha,k}\left(\left(z,[\gamma]\right),w\right)$ are multivalued function on $j_\alpha (\Omega) \cdot (\mathbb{C}-\mathcal{N}_{\alpha (\Omega)})$.
\end{lem}

\begin{proof}
For $\alpha\in SL(2,\mathbb{Z})$,
we set 
$(\Omega_1,(z_1,[\gamma_1]),w_1)=\Big(\alpha(\Omega), \Big(\frac{z}{j_\alpha (\Omega)}, \Big[\frac{\gamma}{j_\alpha(\Omega)}\Big]\Big),\frac{w}{j_\alpha (\Omega)}\Big)$.
We prove the existence of the holomorphic Fourier expansions (\ref{Fourierxi})  of $\xi_s$ by a induction for $s$.

If $s=0$, it is trivial.
If $s=1$,
$\xi_1 (\Omega_1,(z_1,[\gamma_1]),w_1)$
is determined by the integration of the relation
\begin{align}\label{eta_1}
N\frac{\partial}{\partial z_1}\eta_1 (\Omega_1,z_1) = -a_2 (\Omega_1,z_1)
\end{align}
on the arc $\gamma_1\subset \mathbb{C}-\mathcal{N}_{\Omega_1}$
(recall the proof of Theorem \ref{ThmPsi}).
Dividing (\ref{eta_1}) by $j_\alpha (\Omega)^2$,
 considering the relation $\frac{\partial}{\partial z_1}=j_\alpha (\Omega) \frac{\partial}{\partial z}$
and  using the assumption (\ref{Fouriera}),
we have
\begin{align}\label{za2F}
N\frac{\partial}{\partial z} j_\alpha (\Omega)^{-1} \eta_1 \Big(\alpha(\Omega),\frac{z}{j_\alpha (\Omega)}\Big) 
= - \frac{a_2 (\Omega_1,z_1)}{j_\alpha (\Omega)^2} 
= -\sum_{k\geq 0} \tilde{a}_{2,\alpha,k} (z) \exp\left(\frac{2\pi \sqrt{-1} k \Omega}{N_{2,\alpha}}\right).
\end{align}
By integrating (\ref{za2F}) on the arc $\gamma \subset j_\alpha (\Omega) \cdot (\mathbb{C}-\mathcal{N}_{\alpha (\Omega)})$,
we have the holomorphic Fourier expansion at cusps for $\xi_1$ of (\ref{Fourierxi}).
 
 Next,
 assume that we have the holomorphic Fourier expansion (\ref{Fourierxi})  of $\xi_s$ for $s=0,1,\cdots, s_0-1$ $(s_0\geq 1)$.
 We will obtain the holomorphic  Fourier expansion  of  $\xi_{s_0}$.
 By the proof of Theorem \ref{ThmPsi} (1),
 especially the relation (\ref{Induction}),
 $\xi_{s_0} (\Omega_1,(z_1,[\gamma_1]),w_1)$ is given by the integration of the relation 
 \begin{align}\label{etas0Int}
 N\frac{\partial}{\partial z_1}\eta_{s_0} (\Omega_1,z_1) = H_{s_0} \Big(\frac{\partial^\nu}{\partial z_1^\nu} \xi_m (\Omega_1,(z_1,[\gamma_1]),w_1 ), a_l(\Omega_1,z_1) \Big).
 \end{align}
 Here,  $H_{s_0} \Big(\frac{\partial^\nu}{\partial z_1^\nu} \xi_m (\Omega_1,(z_1,[\gamma_1]),w_1 ), a_l(\Omega_1,z_1) \Big)$ 
 is a polynomial in 
 $\frac{\partial^\nu}{\partial z_1^\nu} \xi_m (\Omega_1,(z_1,[\gamma_1]),w_1 )$ $(m\leq s_0-1 , \nu \in \mathbb{Z}_{\geq 0})$ 
 and
 $a_l(\Omega_1,z_1)$ $(l=2,\cdots,N)$.
 By Theorem \ref{ThmPsi} (3),
 the polynomial $H_{s_0}$ is homogeneous of weight $s_0 +1$ 
 with respect to the action of $\Gamma$.
 This implies that,
 by dividing (\ref{etas0Int}) by $j_\alpha (\Omega)^{s_0 +1}$   for $\alpha \in SL(2,\mathbb{Z})$,
the relation
\begin{align} \label{etas0Intmae}
N \frac{\partial}{\partial z} j_\alpha(\Omega)^{-s_0} \eta_{s_0} \Big(\alpha(\Omega),\frac{z}{j_\alpha (\Omega)}\Big)
=H_{s_0}\Big(\frac{\partial^\nu}{\partial z^\nu} j_\alpha (\Omega)^{-m}\xi_m (\Omega_1,(z_1,[\gamma_1]),w_1 ),j_\alpha(\Omega)^{-l} a_l(\Omega,z)\Big),
 \end{align}
holds  similarly to the (\ref{za2F}).
 By the assumption,
 the right hand side of (\ref{etas0Intmae})
 has the holomorphic Fourier expansion at cusps.
 So, by the integration of (\ref{etas0Intmae}) on the arc $\gamma \subset j_\alpha (\Omega) \cdot (\mathbb{C}-\mathcal{N}_{\alpha (\Omega)})$,
 we have the holomorphic Fourier expansion (\ref{Fourierxi}) at cusps for $s_0$.
 
Hence, the assertion is proved.
\end{proof}

\begin{rem}
If $\Gamma=SL(2,\mathbb{Z})$,
 the relation (\ref{Trans}) holds for any $\alpha \in SL(2,\mathbb{Z})$.
 Then,
 from Lemma \ref{Lempoles},
 $j_\alpha (\Omega) \cdot (\mathbb{C}-\mathcal{N}_{\alpha(\Omega)}) = (\mathbb{C}-\mathcal{N}_{\Omega})$
 holds.
 So,  in this case,
 we only need to consider multivalued functions on $\mathbb{C}-\mathcal{N}_\Omega$.
 However, 
 if $\Gamma \not = SL(2,\mathbb{Z})$,
  we need a detailed condition as we saw in Lemma \ref{LemFourier}.  
\end{rem}

\begin{thm}\label{AutoToDiffE}
Suppose $F=\mathbb{Q}$ and $n=1$.
Let $P_{\Omega,z}$  of (\ref{MotherEq}) ($Q_{\Omega,(z,[\gamma])}$ of (\ref{Q}), resp.) 
be differential operators
studied  in Theorem  \ref{AutoToDiffZ}.
Moreover, assume that 
every coefficient $a_l(\Omega,z)$ $(l=2,\cdots,N)$
of the differential operator $P_{\Omega,z}$ of (\ref{MotherEq})
has a holomorphic Fourier expansion at cusps in the form (\ref{Fouriera}).

(1) If $A_j(\Omega)$ $(j=0,-1,\cdots, -M)$ 
are automorphic forms of weight $K+j$ for $\Gamma$,
then any coefficients $b_j(\Omega,(z,[\gamma]))$ of $Q_{\Omega,(z,[\gamma])}$,
which are derived from $A_j(\Omega)$ $(j=0,-1,\cdots, -M)$ in the sense of Theorem \ref{AutoToDiffZ},
give holomorphic functions  of $\Omega \mapsto b_j(\Omega,(z,[\gamma]))$ for generic $(z,[\gamma])$.
Moreover, for $\alpha \in SL(2,\mathbb{Z})$,
every coefficient $b_j(\Omega,(z,[\gamma]))$ of $Q_{\Omega,(z,[\gamma])}$
has a holomorphic Fourier expansion
\begin{align}\label{Fourierb}
j_{\alpha}(\Omega)^{-K-j+M} b_j\Big(\alpha(\Omega),\Big(\frac{z}{j_\alpha (\Omega)}, \Big[ \frac{\gamma}{j_\alpha (\Omega)} \Big]\Big)\Big)
=\sum_{k\geq 0} \tilde{b}_{j,\alpha,k} (z,[\gamma]) \exp \Big(\frac{2\pi \sqrt{-1}k \Omega}{N_{j,\alpha}}\Big),
\end{align}
where $\tilde{b}_{j,\alpha,k} (\Omega,(z,[\gamma]))$ are multivalued analytic function on $ j_{\alpha} (\Omega) \cdot (\mathbb{C}-\mathcal{N}_{\alpha (\Omega)})$.
Furthermore,
 $A_s(\Omega)$ $(s\geq 1)$ are automorphic forms  of weight $K+s$ for $\Gamma$.

(2) Conversely,
we suppose that
 every coefficient $b_j (\Omega,(z,[\gamma]))$ 
of $Q_{\Omega,(z,[\gamma])}$
is a holomorphic function of $\Omega$ 
and
has a holomorphic Fourier expansion  (\ref{Fourierb})  at cusps.
Then,
$A_s(\Omega) $ $(s\geq -M)$,
which are determined in the sense of Theorem \ref{PAs}, 
are automorphic forms for $\Gamma$.
\end{thm}

\begin{proof}
(1) Under the assumption,
as in  the proof of  Theorem \ref{AutoToDiff},
we can see that $\Omega \mapsto b_j(\Omega,(z,[\gamma]))$ are holomorphic for generic $(z,[\gamma])$. 
We prove that $b_j (\Omega,(z,[\gamma]))$ have  holomorphic Fourier expansions (\ref{Fourierb}) for any $\alpha \in SL(2,\mathbb{Z})$.
The coefficients $b_j (\Omega,(z,[\gamma]))$
are determined by the relation (\ref{bbxAx}) inductively.
For $\alpha\in SL(2,\mathbb{Z})$ and $t=0,\cdots, -M$,
we have the holomorphic Fourier expansion
\begin{align}\label{FourierA}
j_\alpha (\Omega)^{-t } A_t(\Omega)=\sum_{k\geq 0} \tilde{A}_{t,\alpha ,k }\exp\Big(\frac{2 \pi \sqrt{-1} \Omega}{N_{t,\alpha}}\Big)
\end{align}
by the assumption.
We set 
$(\Omega_1,(z_1,[\gamma_1]),w_1)=\Big(\alpha(\Omega), \Big(\frac{z}{j_\alpha (\Omega)}, \Big[\frac{\gamma}{j_\alpha(\Omega)}\Big]\Big),\frac{w}{j_\alpha (\Omega)}\Big)$.
From (\ref{bbxAx}),
it holds that
\begin{align}\label{bbxAxF}
b_j(\Omega_1,(z_1,[\gamma_1])) =
H^b_j\Big(  b_m(\Omega_1,(z_1,[\gamma_1])) , \frac{\partial^\nu}{\partial z_1^\nu}  \xi_s(\Omega_1,(z_1,[\gamma_1]),w_1)  ,  A_t(\Omega_1) \Big),
\end{align}
where $H^b_j (  b_m(\Omega_1,(z_1,[\gamma_1])) , \frac{\partial^\nu}{\partial z_1^\nu}  \xi_s(\Omega_1,(z_1,[\gamma_1]),w_1)  ,  A_t(\Omega_1) )$ is a polynomial 
in
$b_m(\Omega_1,(z_1,[\gamma_1]))$ $(m< j)$, $ \frac{\partial^\nu}{\partial z_1^\nu}  \xi_s(\Omega_1,(z_1,[\gamma_1]),w_1) $  $(s,\nu \in \mathbb{Z}_{\geq 0})$ and $A_{t}(\Omega_1)$ $(t = -M ,\cdots,0)$.
By virtue of Theorem \ref{PAs} (3),
(\ref{bbxAxF}) is homogeneous of weight $K+j-M$ with respect to the action of $\Gamma$.
This implies that, by dividing  (\ref{bbxAxF}) by $j_\alpha (\Omega)^{K+j-M}$   for $\alpha \in SL(2,\mathbb{Z})$ 
and considering $\frac{\partial}{\partial z_1}=j_\alpha (\Omega) \frac{\partial}{\partial z}$,
we  obtain
\begin{align}\label{bFourier}
&j_{\alpha} (\Omega)^{-K-j+M}  b_j(\Omega_1,(z_1,[\gamma_1]))  \notag \\
&=
H^b_j\Big( j_{\alpha} (\Omega)^{-K-m+M}  b_m(\Omega_1,(z_1,[\gamma_1])) , \frac{\partial^\nu}{\partial z^\nu} j_\alpha (\Omega)^{-s} \xi_s(\Omega_1,(z_1,[\gamma_1]),w_1)  , j_{\alpha} (\Omega)^{-t} A_t(\Omega_1) \Big).
\end{align}
By virtue of the assumption and Lemma \ref{LemFourier},
$  \frac{\partial^\nu}{\partial z^\nu} j_\alpha (\Omega)^{-s} \xi_s(\Omega_1,(z_1,[\gamma_1]),w_1)$ ( $ j_{\alpha} (\Omega)^{-t} A_t(\Omega_1) $ , resp.)  have Fourier expansions (\ref{Fourierxi}) ((\ref{FourierA}), resp.).
So, we can inductively obtain the Fourier expansions (\ref{Fourierb}) of $j_{\alpha} (\Omega)^{-K-j+M}  b_j(\Omega_1,(z_1,[\gamma_1])) $.

Next,
 we will consider the Fourier expansion of $A_s(\Omega)$ for $s\geq 1$.
By (\ref{AxAx}),
we obtain 
\begin{align}\label{Ha}
A_s(\Omega_1) = H^{a}_s \Big(  b_j(\Omega_1,(z_1,[\gamma_1])) , \frac{\partial^\nu}{\partial z_1^\nu}  \xi_s(\Omega_1,(z_1,[\gamma_1]),w_1)  ,  A_t(\Omega_1) \Big),
\end{align}
where $H^{a}_s \Big(  b_j(\Omega_1,(z_1,[\gamma_1])) , \frac{\partial^\nu}{\partial z_1^\nu}  \xi_s(\Omega_1,(z_1,[\gamma_1]),w_1)  , j_{\alpha} (\Omega)^{-t} A_t(\Omega_1) \Big)$ is a polynomial in 
$b_j(\Omega_1,(z_1,[\gamma_1])) $ $(0\leq j \leq M)$, 
$\frac{\partial^\nu}{\partial z_1^\nu}  \xi_s(\Omega_1,(z_1,[\gamma_1]),w_1)$ $(s,\nu \in \mathbb{Z}_{\geq 0})$
and
$A_t(\Omega_1)$ $(t<s)$.
We can see that the polynomial is homogeneous of weight $s$ under the action of $SL(2,\mathbb{Z})$ also.
Therefore, dividing (\ref{Ha}) by $j_\alpha (\Omega)^{s}$ ($\alpha \in SL(2,\mathbb{Z})$) and using  $\frac{\partial}{\partial z_1}=j_\alpha (\Omega) \frac{\partial}{\partial z}$,
we have
\begin{align}
&j_{\alpha}(\Omega)^{-s}  A_s(\Omega_1) \notag \\
&= H^{a}_s \Big( j_{\alpha} (\Omega)^{-K-j+M}  b_j(\Omega_1,(z_1,[\gamma_1])) , \frac{\partial^\nu}{\partial z^\nu} j_\alpha (\Omega)^{-s} \xi_s(\Omega_1,(z_1,[\gamma_1]),w_1)  , j_{\alpha} (\Omega)^{-t} A_t(\Omega_1) \Big).
\end{align}
So, 
we can also obtain the Fourier expansions of $j_{\alpha}(\Omega)^{-s}  A_s(\Omega) $ inductively.

(2)
We only need to obtain the holomorphic Fourier expansions of $A_s(\Omega)$ $(s\geq -M)$ at cusps.
By the
same argument with  the latter of the proof of (1),
we can obtain the Fourier expansion of $j_\alpha(\Omega)^{-s}A_s(\Omega)$ inductively
from the Fourier expansion (\ref{Fourierb}).
\end{proof}

\subsection{A formulation via a ring of generating functions}

In this subsection,
we will give an interpretation of Theorem \ref{PAs}, \ref{AutoToDiffZ}, \ref{AutoToDiff} and \ref{AutoToDiffE}  using a ring of generating functions
for sequences of automorphic forms.

Let $\Gamma \subset Sp(n,F)$
be a congruence subgroup
and
$\mathcal{M}_K (\Gamma)$ be the vector space of automorphic forms for $\Gamma$ of weight $K$.
It is well-known that $\mathcal{M}_K (\Gamma) =\{0\}$ if $K<0.$
Let $\mathcal{M}(\Gamma) = \displaystyle \bigoplus _{K=0}^\infty \mathcal{M}_K (\Gamma)$ be the graded ring of automorphic forms for $\Gamma$.
Let $V$ be the ring of formal Laurent series in $\lambda^{-1}$ over $\mathcal{M}(\Gamma)$:
$
V=\Big\{\displaystyle \sum_{s=-M}^\infty A_s(\Omega) \lambda^{-s} \Big| M\in \mathbb{Z}_{\geq 0}, A_s(\Omega) \in \mathcal{M}(\Gamma)  \Big\}.
$
We take a subspace $R_K $ of $V$ defined by
$
R_K= \Big\{ \displaystyle \sum_{s=-M}^\infty A_s(\Omega) \lambda^{-s} \in V \Big| A_s(\Omega) \in \mathcal{M}_{s+K} (\Gamma)  \Big\}.
$
Then,
$\displaystyle
R= \bigoplus_{K=0}^\infty R_K
$
is a  graded ring.
The ring $R$ can be regarded as a ring of generating functions for  sequences $\{A_s(\Omega)\}$ $(A_s(\Omega)\in \mathcal{M}_{s+K} (\Gamma))$ of automorphic forms.
We note that $R_0$ gives a subring of $R$.

Now,
we define a vector space $\mathcal{D}^P_K$ and a ring $\mathcal{D}^P$ of differential operators.
Let $P=P_{\Omega,z}$ of (\ref{MotherEq}) be a differential operator of weight $N$ with respect to $\Gamma$.
However,
if $F=\mathbb{Q}$ and $n=1$,
we additionally assume that the coefficients $a_l(\Omega,z)$ $(l=0,\cdots,N)$
have holomorphic Fourier expansions  (\ref{Fouriera}) at cusps.

\begin{df}\label{dfDP}
For a fixed congruence subgroup $\Gamma(\subset Sp(n,F))$, set
$$
\tilde{\mathcal{D}}_K^P=\{Q=Q_{\Omega,(z,[\gamma])}\ | Q \text{ is given by } (\ref{Q}); Q \text{ commutes with } P; Q \text{ is of weight } K \text{ with respect to } \Gamma  \}.
$$
Then,

(i)  if $F\not = \mathbb{Q}$ or $n\geq 2$,
set
$$
\mathcal{D}_K^P =\{Q\in \tilde{\mathcal{D}}_K^P| \Omega \mapsto b_j(\Omega,(z,[\gamma])) \text{ are holomorphic for generic} (z,[\gamma]) \}.
$$

(ii) if $F = \mathbb{Q}$ and $n =1$,
set
\begin{align*}
\mathcal{D}_K^P =\{ Q\in \tilde{\mathcal{D}}_K^P| 
& \Omega \mapsto b_j(\Omega,(z,[\gamma])) \text{ are holomorphic for  generic} (z,[\gamma]) ; \\
&b_j(\Omega,(z,[\gamma])) \text { have holomorphic Fourier expansions } (\ref{Fouriera}) \text{ at cusps} \}.
\end{align*}

Set $\displaystyle \mathcal{D}^P=\bigoplus_{K=0}^\infty \mathcal{D}_K^P.$
This is a commutative graded ring (see Theorem \ref{PAs} (2)).
\end{df}

Let $\chi: \mathcal{D}^P \rightarrow R$ be a mapping 
given by
\begin{align}\label{chi1}
Q_{\Omega,(z,[\gamma])} = Q \mapsto A =A (\Omega,\lambda)
\end{align}
if $Q\Psi = A\Psi$ for $\Psi = \Psi(\Omega,(z,[\gamma]),w,\lambda) $ of (\ref{Psi}) for generic $\lambda$.
 Theorem \ref{AutoToDiffZ}  implies that $\chi$ is an injective mapping.
Moreover,
if $Q_1,Q_2 \in \mathcal{D}^P$  and $\chi(Q_j) = A_j $ $(j=1,2)$, 
we have
$(Q_1+Q_2) \Psi = (A_1+A_2)\Psi$ and $(Q_1 Q_2)\Psi = A_1 A_2 \Psi.$
So, 
by a similar argument to the end of the proof of Proposition \ref{PAsz},
we can see that
$\chi (Q_1+Q_2)=A_1+A_2$, $\chi(Q_1 Q_2)=A_1 A_2 $ and $\chi (1)=1$ hold.
Namely,
$\chi$ gives an embedding $\mathcal{D}^P \hookrightarrow R$ of rings.

\begin{df}\label{dfSP}
Let $S_K ^P$ be the vector space  $ \chi (\mathcal{D}_K^P) (\subset R)$  over $\mathbb{C}$. 
Let  $\displaystyle S^P$ be the graded ring $\displaystyle \chi(\mathcal{D}^P)=\bigoplus_{K=0}^\infty S_K^P$.
\end{df}

For any $A=\displaystyle \sum_{s=-M}^\infty A_s(\Omega) \lambda^{-s}\in S_K^P$ ($A_{-M}(\Omega)\not\equiv 0$),
we put
\begin{align}\label{prin}
Prin (A)=A_{-M}(\Omega)\lambda^{M} + \cdots + A_0(\Omega), \quad \quad (A_{-M}(\Omega)\not\equiv 0).
\end{align}
We call $Prin (A)$ the principal part of $A$.
Set $W_K=\displaystyle Prin(S_K^P)=\bigoplus_{s=0}^K  \mathcal{M}_{K-s}'(\Gamma) \lambda^s$.
Here, $\mathcal{M}_{K-s}' (\Gamma) $ is a subspace of $\mathcal{M}_{K-s}(\Gamma)$.
We have the following reformulation of  our main results.

\begin{thm}\label{ThmRef}
The mapping $\chi:\mathcal{D}^P\rightarrow S^P$ given by (\ref{chi1}) is an isomorphism of graded rings.
For fixed $K$, 
the mapping 
$Prin: S^P_K \rightarrow W_K$ of (\ref{prin}) is an isomorphism of vector spaces over $\mathbb{C}$.
Especially, the sequence
\begin{align*}
\mathcal{D}_K^P \xrightarrow{\chi} S_K^P \xrightarrow{Prin} W_K
\end{align*}
gives  isomorphisms of three vector spaces.
   \end{thm}

\begin{proof}
Due to Theorem \ref{AutoToDiff} and  \ref{AutoToDiffE},
 $Prin$ of (\ref{prin}) is a bijective mapping.
So, $Prin$ is also an isomorphism of vector spaces over $\mathbb{C}$.
\end{proof}

\begin{rem}
We can naturally define $Prin$ on $\displaystyle S^P=\bigoplus_{K=0}^\infty S_K^P.$
However, this does not give a homomorphism of rings.
\end{rem}

From the proof of  Theorem \ref{AutoToDiffZ} and \ref{AutoToDiff}, we can see the following.

\begin{cor}
An operator $Q\in \mathcal{D}_K^P$ is of rank $M$ if and only if $Prin (\chi (Q))$ is given by the form (\ref{prin}).
 Moreover, the leading coefficient of $Q\in \mathcal{D}_K^P$ is a constant number $c$
 if and only if $M=K$ and
 $A_{-K}(\Omega)\in \mathcal{M}_0(\Gamma)$ is given by $A_{-K}(\Omega) \equiv c.$
\end{cor}

Let $\mathcal{D}_{K,M}^P$ be the subspace consisting of differential operators $Q$ whose ranks are at most  $M$.
We have $\mathbb{C}=\mathcal{D}_{K,0}^P \subset \mathcal{D}_{K,1}^P \subset \cdots \subset \mathcal{D}_{K,K}^P=\mathcal{D}_K^P.$

From Theorem \ref{ThmRef},
any element of 
$Q\in \mathcal{D}_{K,M}^P$
is parametrized by
the elements of 
the vector space
$\mathcal{M}_{K-M}' (\Gamma) \oplus \cdots \oplus \mathcal{M}_K' (\Gamma).$ 
Then, $\mathcal{D}_{K,M}^P$ has $\displaystyle \sum_{s=0}^M {\rm dim}\mathcal{M}_{K-s}' (\Gamma)$ complex parameters.

\begin{cor}\label{CorDim}
$$
{\rm dim}_\mathbb{C} \mathcal{D}_{K,M}^P = \bigoplus_{s=0}^M {\rm dim}_{\mathbb{C}} \mathcal{M}_{K-s}' (\Gamma).
$$
Especially,
$$
{\rm dim}_\mathbb{C} \mathcal{D}_{K}^P = \bigoplus_{s=0}^K {\rm dim}_{\mathbb{C}} \mathcal{M}_{K-s}' (\Gamma). 
$$
\end{cor}

Anyway,
if a differential operator $P_{\Omega,z}$  of weight $N$ with respect to the action of $\Gamma $ 
and automorphic forms $A_j(\Omega)\in \mathcal{M}_{K+j}' (\Gamma)$ $(j=0,\cdots,-M)$
are given,
there is the unique differential operator $Q_{\Omega,(z,[\gamma])}=Q \in \mathcal{D}_{K,M}^P$ which commutes with $P = P_{\Omega,z}$.
We set
$$
Q_{\Omega,(z,[\gamma])}(P_{\Omega,z};A_{-M}(\Omega) \lambda^M + \cdots+ A_0(\Omega)) = \chi^{-1} \circ Prin^{-1}(A_{-M}(\Omega) \lambda^M + \cdots+ A_0(\Omega)).
$$

\begin{exap}\label{ExapLameSpecial}
In this example,
we consider the special case for $B=2$ of the Lam\'e operator of (\ref{Lame}):
\begin{align}\label{Lame1}
P_{\Omega,z}=\frac{\partial^2 }{\partial z^2} - 2\wp(\Omega,z).
\end{align}
As we saw in Example \ref{ExapLame},
$P_{\Omega,z}$ is of weight $2$ for $SL(2,\mathbb{Z}).$

Automorphic forms for $\Gamma =SL(2,\mathbb{Z})$ are called elliptic modular forms.
Let $\mathcal{M}_k (\Gamma)$ be the vector space of elliptic modular forms of weight $k$.
According to Theorem \ref{AutoToDiff} and \ref{AutoToDiffE},
elliptic modular forms $A_j(\Omega)\in \mathcal{M}_{K+j} (\Gamma)$ $(j\in \{0,\cdots,-M\})$
 determine a differential operator $Q_{\Omega,(z,[\gamma])}$ of rank $M$ of weight $K$ with respect to $\Gamma$,
 where $Q_{\Omega,(z,[\gamma])}$ commutes with $P_{\Omega,z}.$
 
 In the following, we shall consider a simple case of $K=3$ and $M=3$.
 In this case, modular forms must be quite simple:
 $A_{-3}(\Omega)\equiv {\rm const} \in \mathcal{M}_0 (\Gamma)$ and 
 $ A_{-j}(\Omega)\equiv 0 \in \mathcal{M}_{3+j} (\Gamma)$ $(j=0, 1,2)$,
 because we have
 \begin{align}\label{ElModDim}
 \begin{cases}
 &\mathcal{M}_0(\Gamma)=\mathbb{C}, \\
  &\mathcal{M}_k (\Gamma)=\{0\} \quad (k=1,2,3)
 \end{cases}
 \end{align}
  (for detail, see \cite{S71}).
 So, let us obtain the differential operator $Q_{\Omega,(z,[\gamma])}=Q_{\Omega,(z,[\gamma])}(P_{\Omega,z}; \lambda^3)=Q_{\Omega,(z,[\gamma])}(P_{\Omega,z}; \lambda^3+0\lambda^2+0\lambda+0)$.
 
 First, we calculate $\Psi $ of (\ref{Psi}).
 As we saw in the proof of Proposition \ref{PropPsiz} and  Theorem \ref{ThmPsi},
 we can determine $\{\xi_s(\Omega,(z,[\gamma]),w)\}$ inductively.
 Taking $w=\frac{1}{2}\not \in \mathbb{Z} + \mathbb{Z} \Omega$,
  we have in fact
 \begin{align}\label{LameXi}
 \begin{cases}
 & \xi_0(\Omega,(z,[\gamma]),\frac{1}{2}) =1,\\
 & \xi_1 (\Omega,(z,[\gamma]),\frac{1}{2}) =-\zeta(\Omega,z) + \zeta_1 (\Omega),\\
 & \xi_2 (\Omega,(z,[\gamma]),\frac{1}{2}) = \frac{1}{2} \zeta^2(\Omega,z) +\frac{1}{2} \zeta_1^2(\Omega) - \zeta_1(\Omega) \zeta(\Omega,z) -\frac{1}{2} \wp(\Omega,z)+\frac{1}{2}\wp_1(\Omega),\\
 & \cdots ,
 \end{cases}
 \end{align}
 where $\zeta (\Omega,z)$ is the Weierstrass $\zeta$-function
 $$
 \zeta(\Omega,z)=\frac{1}{z} + \displaystyle \sum_{(n_1,n_2)\in \mathbb{Z}^2 -\{(0,0)\}} \Big(\frac{1}{z-n_1-n_2\Omega} +\frac{1}{n_1 +n_2 \Omega} +\frac{z}{(n_1 + n_2 \Omega)^2} \Big)
 $$
 and $\zeta_1(\Omega)=\zeta(\Omega,\frac{1}{2})$ and $\wp_1(\Omega)=\wp(\Omega,\frac{1}{2})$.
 (We note  that  the right hand side of (\ref{LameXi})  satisfy the transformation law (\ref{XiTrans}).)
 
 From our data, set
 $A(\Omega,\lambda)=Prin^{-1} (\lambda^3)=\displaystyle \sum_{s=-3} ^\infty A_s(\Omega) \lambda^{-s} = \lambda^3 + 0 \lambda^2 + 0 \lambda + 0+A_1(\Omega) \lambda^{-1}+\cdots$.
 For $\Psi(\Omega,(z,[\gamma]),w,\lambda)$ given by (\ref{LameXi}),
 we can  uniquely find the differential operator 
 $Q_{\Omega,(z,[\gamma])} = \displaystyle \sum_{j=0}^3 b_j(\Omega,(z,[\gamma])) \frac{\partial ^{3-j}}{\partial z^{3-j}}$
 satisfying 
 $Q_{\Omega,(z,[\gamma])} \Psi (\Omega,(z,[\gamma]),w,\lambda)=A(\Omega,\lambda) \Psi (\Omega,(z,[\gamma]),w,\lambda)$.
In fact,
by a direct calculation as in the proof of Theorem \ref{AutoToDiffZ},
we can obtain
 $
 b_0(\Omega,(z,[\gamma])) =1,
 b_1(\Omega,(z,[\gamma]))  = 0, 
 b_2(\Omega,(z,[\gamma]))  = -3 \wp(\Omega,z) $ and  
 $
 b_3(\Omega,(z,[\gamma]))  = -\frac{3}{2} \frac{\partial}{\partial z}\wp(\Omega,z).
 $
 Therefore,
 \begin{align}\label{LameQ}
 Q_{\Omega,(z,[\gamma])}
= Q_{\Omega,(z,[\gamma])}(P_{\Omega,z}; \lambda^3)
  = \frac{\partial^3}{\partial z^3} -3 \wp(\Omega,z) \frac{\partial}{\partial z} -\frac{3}{2} \Big( \frac{\partial}{\partial z}\wp (\Omega,z) \Big)
 \end{align}
 is the differential operator we want.
 This is of weight $3$ with respect to $\Gamma =SL(2,\mathbb{Z})$.
 
 From Theorem \ref{ThmRef}, Corollary  \ref{CorDim} and the fact (\ref{ElModDim}),
 $Q_{\Omega,(z,[\gamma])}$ of (\ref{LameQ})
 is the unique element of $\mathcal{D}_3^P$ for $P=P_{\Omega,z}$ of (\ref{Lame1})
 up to a constant factor.
  
 We remark that 
 the relation between such a Lam\'e operator $P_{\Omega,z}$ of (\ref{Lame1}) and $Q_{\Omega,(z,[\gamma])}$  of (\ref{LameQ}) 
 were precisely studied from the viewpoint of integrable systems or physics (see \cite{W}, \cite{DMN} or \cite{Mulase}).
 Our result gives a new interpretation on this topic from the viewpoint of elliptic modular forms.
\end{exap}

\subsection{The family of algebraic curves $\mathcal{R}_\Omega$}

For  $X\in \mathbb{C}=\mathbb{P}^1 (\mathbb{C})-\{\infty\}$
and $P_{\Omega,z}$ of (\ref{MotherEq}),
we consider the  differential equation $P_{\Omega,z} u = Xu$ 
and its space of solutions $\mathcal{L}(P_{\Omega,z},X)$.
Suppose $Q_{\Omega,(z,[\gamma])}$ of (\ref{Q}) commutes with $P_{\Omega,z}$.
Letting $\lambda_1,\cdots,\lambda_N$ be the disjoint solutions of $\lambda^N = X$,
$\Psi(\Omega, (z,[\gamma]),w,\lambda_j)$ gives an eigenvector for the eigenvalue $A(\Omega,\lambda_j)$.
Let $\mathcal{Q}_{\Omega,[\gamma],X}$ be the linear operator derived from $Q_{\Omega,(z,[\gamma])}$ on $\mathcal{L}(P_{\Omega,z},X)$. 
The characteristic  polynomial of $\mathcal{Q}_{\Omega,[\gamma],X}$ is given by
\begin{align}\label{eigen}
\prod_{j=1}^N (Y-A(\Omega,\lambda_j)).
\end{align}
Due to Corollary \ref{LemCj},
(\ref{eigen}) gives a polynomial
$F_\Omega (X,Y) $ in $X$ and $Y$.
From Theorem
\ref{ThmRiemannSurfacez},
it follows that
$F_\Omega (P_{\Omega,z},Q_{\Omega,(z,[\gamma])})=0$.
We set
\begin{align}\label{FOmega}
F_\Omega (X,Y) = \displaystyle \sum_{j,k} f_{j,k} (\Omega) X^j Y^k .
\end{align}

\begin{thm}\label{ThmRiemannSurface}
Let $P=P_{\Omega,z}$ be the differential operator of (\ref{MotherEq}).
Take $Q=Q_{\Omega,(z,[\gamma])} \in \mathcal{D}_K^P.$
Then,
the coefficient
$f_{j,k} (\Omega)$ in (\ref{FOmega})
is an automorphic form of weight $NK-Nj-Kk$ for $\Gamma$.
\end{thm}

\begin{proof}
From Theorem \ref{ThmRef},
we suppose that
$Q\in \mathcal{D}_K^P$ is given by
$Q=Q_{\Omega,(z,[\gamma])} (P_{\Omega,z};A_{-M}(\Omega)\lambda^M + \cdots+A_0(\Omega))$.
Then,
$\chi(Q)=Prin^{-1} (A_{-M}(\Omega)\lambda^M + \cdots+A_0(\Omega))$ 
is given by
 a series $A(\Omega,\lambda)=\displaystyle \sum_{s=-M}^\infty A_s(\Omega) \lambda^{-s}$.
Due to Theorem \ref{AutoToDiff} and \ref{AutoToDiffE},
$A_s(\Omega)$ $(s\geq 1)$ are automorphic forms of weight $s + K$.
Since the set of automorphic forms is a ring,
together with the argument in Section 1.3,
the coefficients of
the polynomial
$\displaystyle \prod _{j=1}^N (Y-A(\Omega,\lambda_j)) $
in $X $ and $Y$,
where $\lambda_j^N =X$,
are automorphic forms for $\Gamma$.

Since $P_{\Omega,z}$ ($Q_{\Omega,(z,[\gamma])}$) is of weight $N$ ($K$, resp.) for the action of $\Gamma$,
we have the action of $\alpha\in \Gamma$
given by
$(\Omega,X,Y)\mapsto (\alpha (\Omega), j_\alpha (\Omega)^N X, j_\alpha (\Omega)^K Y)=(\alpha(\Omega),X_1,Y_1).$
When we describe $F_\Omega (X,Y)$ as in (\ref{FOmega}),
we have
\begin{align*}
f_{j,k} (\alpha (\Omega))=j_\alpha (\Omega)^{NK-Nj-Kk} f_{j,k}(\Omega).
\end{align*}
Hence, by comparing the coefficients,
$f_{j,k} (\Omega)$ is of weight $NK-Nj-Kk.$
\end{proof}

By Theorem \ref{ThmRef} and Theorem \ref{ThmRiemannSurface},
the family $\{F_\Omega (X,Y)=0 | \Omega \in \mathbb{H}_n^g\}$ of algebraic curves is uniquely determined by the differential operator $P_{\Omega,z}$ and given principal part
$A_{-M}(\Omega)\lambda^M + \cdots + A_0(\Omega)$ of a Laurent series $Prin^{-1}(A_{-M}(\Omega)\lambda^M + \cdots + A_0(\Omega))\in S_K^P$. 
We denote such a family by
$\mathcal{F}(P_{\Omega,z};A_{-M}(\Omega)\lambda^M + \cdots + A_0(\Omega)) $ whose members are $\mathcal{R}_{\Omega}=\mathcal{R}_\Omega (P_{\Omega,z};A_{-M}(\Omega)\lambda^M + \cdots + A_0(\Omega)).$

\begin{exap}\label{ExapLameRiemann}
Let $P_{\Omega,z}$ ($Q_{\Omega,(z,[\gamma])}=Q_{\Omega,(z,[\gamma])}(P_{\Omega,z}:\lambda^3)$, resp.) be the operator of (\ref{Lame1}) ((\ref{LameQ}), resp.),
 as we saw in Example \ref{ExapLameSpecial}.
We note that the $\wp$-function satisfies the Weierstrass equation
\begin{align}\label{WEQ}
\Big(\frac{\partial}{\partial z}\wp(\Omega,z)\Big)^2 = 4 \wp ^3 (\Omega,z) -g_2(\Omega) \wp(\Omega,z) -g_3 (\Omega),
\end{align}
 where
 $g_2(\Omega)=\displaystyle \sum_{(n_1,n_2)\in \mathbb{Z}^2-\{(0,0)\}}\frac{60}{(n_1 + n_2 \Omega)^4}$ and $g_3(\Omega)=\displaystyle \sum_{(n_1,n_2)\in \mathbb{Z}^2-\{(0,0)\}}\frac{140}{(n_1 + n_2 \Omega)^6}.$
 It is well-known that $g_2(\Omega) \in \mathcal{M}_4(\Gamma)$ and $g_3(\Omega)\in \mathcal{M}_6 (\Omega)$ (for detail, see \cite{S71}).
Using the relation (\ref{WEQ}),   
 we can see that 
 the defining equation $F_\Omega (X,Y)=0$ of $\mathcal{R}_\Omega=\mathcal{R}_\Omega (P_{\Omega,z};\lambda^3)$ is given by
 $$
 F_\Omega (X,Y) = Y^2 -X^3 - \frac{g_2(\Omega)}{4} X -\frac{g_3 (\Omega)}{4}.
 $$
 So,   $f_{0,2}=f_{3,0}(\Omega)=1, f_{1,0} (\Omega) = \frac{g_2(\Omega)}{4}$ and $f_{0,0}(\Omega)=\frac{g_3(\Omega)}{4}$.
 In this case,
 the family
$\mathcal{F}(P_{\Omega,z};\lambda^3)$ consists of non-singular algebraic curves of genus $1$.
\end{exap}

Let $\pi_\Omega:\mathcal{R}_\Omega \rightarrow \mathbb{P}^1(\mathbb{C})$ be the canonical projection
given by $(X,Y) \mapsto X$.

\subsection{A criterion for single-valued differential operators with actions of $\Gamma$}

In this subsection, 
we use the same notation with that of Section 2.4 and Section 2.5.
Moreover,
we assume
\begin{align}\label{assume}
\text{ there exists } s \hspace{1mm} ( s \geq -M),
\text{ where }  N \text{ and } s \text{ are coprime}, 
\text{ such that } A_s (\Omega)\not \equiv 0
\end{align}
for $\{A_s(\Omega)\}$ of (\ref{starstar}).

\begin{rem}
There are so many cases
that the condition (\ref{assume}) holds.
For example,
if $N$ and $M$ are coprime and $Q_{\Omega,(z,[\gamma])}$ is given by
$Q_{\Omega,(z,[\gamma])} (P_{\Omega,z}; \lambda^M + A_{-M+1} (\Omega) \lambda^{M-1} + \cdots + A_0(\Omega))$
(namely the case of $A_{-M}(\Omega)\equiv 1$),
the condition (\ref{assume}) is satisfied. 
\end{rem}

By a similar argument as in Section 1.4,
we have the eigenfunction 
\begin{align}\label{psi}
\psi(\Omega,(z,[\gamma]),w,p)=\displaystyle \sum_{l=0}^{N-1} h_l(\Omega,w,p) C_l(\Omega,(z,[\gamma]),w,p)
\end{align}
of the operator $\mathcal{Q}_{\Omega,[\gamma],X}$ on $\mathcal{L}(P_{\Omega,z},X)$.
Here,
$C_l(\Omega,(z,[\gamma]),w,X) \in \mathcal{L} (P_{\Omega,z},X)$
such that
\begin{align}\label{Cl}
\frac{\partial^r}{\partial z^r} C_l(\Omega,(z,[\gamma]),w,X)\Big|_{(z,[\gamma])=(w,[id])} =\delta_{r,l}.
\end{align}

We note that 
the function $C_l(\Omega,(z,[\gamma]),w,X)$ of (\ref{Cl}) satisfies
the transformation law
\begin{align}\label{ClTrans}
C_l(\Omega_1,(z_1,[\gamma_1]),w_1,X_1) = \frac{1}{j_\alpha (\Omega)^l} C_l(\Omega,(z,[\gamma]),w,X),
\end{align}
where $\alpha \in \Gamma$ and  $(\Omega_1,(z_1,[\gamma_1]),w_1,X_1) $ is given in (\ref{11111}),
because
the equation $P_{\Omega,z} u = X u$ 
coincides with $P_{\Omega_1,z_1} u = X_1 u$
under the transformation $(\Omega,(z,[\gamma]),w,X)\mapsto (\Omega_1,(z_1,[\gamma_1]),w_1,X_1)$
and it holds that
$$
\frac{\partial^r}{\partial z^r}C_l(\Omega_1,(z_1,[\gamma_1]),w_1,X_1) = \Big(\frac{d z_1}{dz} \Big)^r \frac{\partial^r}{\partial z_1^r} C_l(\Omega_1,(z_1,[\gamma_1]),w_1,X_1) = \frac{1}{j_\alpha (\Omega)^l} \delta_{r, l}.
$$
If $p=(X,Y) \in \mathcal{R}_\Omega$, 
set $p_1=(X_1,Y_1)=(j_\alpha (\Omega)^N X, j_\alpha (\Omega)^M Y)$.
From (\ref{FOmega}),
$(X_1,Y_1)\in \mathcal{R}_{\Omega_1}$.
The vector ${}^t(h_0(\Omega,w,p),\cdots,h_{N-1}(\Omega,w,p))$ 
admits a transformation law
\begin{align}
\begin{pmatrix}
h_0(\Omega_1,w_1,p_1)\\
h_1(\Omega_1,w_1,p_1)\\
\cdots \\
h_{N-1}(\Omega_1,w_1,p_1)
\end{pmatrix}
=
\begin{pmatrix}
h_0(\Omega,w,p)\\
j_\alpha (\Omega) h_1(\Omega,w,p)\\
\cdots \\
j_\alpha(\Omega)^{N-1} h_{N-1}(\Omega,w,p)
\end{pmatrix}.
\end{align}

\begin{thm}\label{ThmCriteriaAuto}
Assume that
$N$ is  a prime number
and
the differential operators $P_{\Omega,z}$ of (\ref{MotherEq})
and $Q_{\Omega,(z,[\gamma])}$ of (\ref{Q})
satisfy the condition (\ref{assume}).
Suppose the arithmetic genus  $\varpi(\mathcal{R}_\Omega)$
is smaller than $N$ for any $\Omega \in \mathbb{H}_n^g$.
Then all coefficients of $Q_{\Omega,(z,[\gamma])}$ are single-valued functions of $z$.
\end{thm}

\begin{proof}
As in  Theorem \ref{Thmphipolez} and Corollary \ref{CorHantei},
the function
$\mathcal{R}_\Omega  \ni p \mapsto \psi (\Omega,(z,[\gamma]),w,p) \in \mathbb{P}^1(\mathbb{C})$
has at most
$\varpi (\mathcal{R}_\Omega)$
 poles.
 By our assumption,
 we have
 $\varpi (\mathcal{R}_\Omega) <N$ for any $\Omega$.
By a similar argument to the proof of Theorem \ref{ThmCriteria},
we can prove that
every coefficient $Q_{(\Omega,(z,[\gamma]))}$  are single-valued in $z$.
\end{proof}

The phenomenon
that we saw in Example \ref{ExapLame}, \ref{ExapLameSpecial} and \ref{ExapLameRiemann}
gives a typical example of the criterion of Theorem \ref{ThmCriteriaAuto}.
Namely,
the rank of the Lam\'e operator
$P_{\Omega,z}$ of (\ref{Lame1})
is the prime number $N=2$,
the commutative operator $Q_{\Omega,(z,[\gamma])}$ of (\ref{LameQ}) of rank $M=3$
is single-valued,
where
$P_{\Omega,z}$ and $Q_{\Omega,(z,[\gamma])}$
 give a point of
the non-singular curve 
$\mathcal{R}_\Omega \in \mathcal{F}(P_{\Omega,z};\lambda^3)$
of genus $1<2=N$.

\section*{Acknowledgment}
This work is supported by 
The JSPS Program for Advancing Strategic International Networks to Accelerate the Circulation of Talented Researchers
``Mathematical Science of Symmetry, Topology and Moduli, Evolution of International Research Network based on OCAMI''
and
The Sumitomo Foundation Grant for Basic Science Research Project (No.150108).

\begin{center}
\hspace{8.8cm}\textit{Atsuhira  Nagano}\\
\hspace{8.8cm}\textit{Department of Mathematics}\\
\hspace{8.8cm}\textit{King's College London}\\
\hspace{8.8cm}\textit{Strand, London, WC2R 2LS}\\
\hspace{8.8cm}\textit{The United Kingdom}\\
 \hspace{8.8cm}\textit{(E-mail: atsuhira.nagano@gmail.com)}
  \end{center}

\end{document}